\documentclass[11pt]{amsart}

\usepackage{subfiles}

\usepackage{comment}

\usepackage{float}
\usepackage{marginnote}
\usepackage{tabu}
\usepackage[margin=1.05in]{geometry}
\usepackage[dvipsnames]{xcolor}   		             		
\usepackage{graphicx}			
\usepackage{amssymb}
\usepackage{mathrsfs}
\usepackage{amsthm}
\usepackage{amsmath}
\usepackage{stmaryrd}
\usepackage{tikz}
\usepackage{tikz-cd}
\usepackage{accents}
\usepackage{upgreek}
\usepackage{enumerate}
\usepackage{bm}
\usepackage{mathtools}
\usepackage[all]{xy}
\usepackage{caption}

\usepackage{url}
\usepackage{float}
\usepackage{todonotes}

\usepackage{colonequals}
\usepackage{bbm}

\usepackage{longtable}

\usepackage[full]{textcomp}
\usepackage[cal=cm]{mathalfa}
\usepackage{xparse}
\usepackage{comment}
\usepackage{cite}

\tikzset{
  commutative diagrams/.cd, 
  arrow style=tikz, 
  diagrams={>=stealth}
}

\theoremstyle{theorem}
\newtheorem{innercustomthm}{Theorem}
\newenvironment{customthm}[1]
  {\renewcommand\theinnercustomthm{#1}\innercustomthm}
  {\endinnercustomthm}
  
\makeatletter
\def\@tocline#1#2#3#4#5#6#7{\relax
  \ifnum #1>\c@tocdepth 
  \else
    \par \addpenalty\@secpenalty\addvspace{#2}%
    \begingroup \hyphenpenalty\@M
    \@ifempty{#4}{%
      \@tempdima\csname r@tocindent\number#1\endcsname\relax
    }{%
      \@tempdima#4\relax
    }%
    \parindent\z@ \leftskip#3\relax \advance\leftskip\@tempdima\relax
    \rightskip\@pnumwidth plus4em \parfillskip-\@pnumwidth
    #5\leavevmode\hskip-\@tempdima
      \ifcase #1
       \or\or \hskip 1em \or \hskip 2em \else \hskip 3em \fi%
      #6\nobreak\relax
    \dotfill\hbox to\@pnumwidth{\@tocpagenum{#7}}\par
    \nobreak
    \endgroup
  \fi}
\makeatother

\usetikzlibrary{calc}
\usetikzlibrary{fadings}
\usetikzlibrary{decorations.pathmorphing}
\usetikzlibrary{decorations.pathreplacing}
\usepackage{tikz,tikz-cd,tikz-3dplot}
\usepackage{pgfplots}
\usetikzlibrary{arrows,shadows,positioning, calc, decorations.markings, 
hobby,quotes,angles,decorations.pathreplacing,intersections,shapes}
\usepgflibrary{shapes.geometric}
\usetikzlibrary{fillbetween,backgrounds}

\tikzset{
  arrow/.pic={\path[tips,every arrow/.try,->,>=#1] (0,0) -- +(0,4pt);},
  pics/arrow/.default={triangle 90}
}

\tikzset{->-/.style={decoration={
  markings,
  mark=at position .6 with {\arrow{latex}}},postaction={decorate}}
  }

\tikzset{
  c/.style={every coordinate/.try}
}

\newcounter{marginnote}
\DeclareMathAlphabet{\mathpzc}{OT1}{pzc}{m}{it}

\usepackage[backref=page]{hyperref}
\hypersetup{
  colorlinks   = true,          
  urlcolor     = blue,          
  linkcolor    = purple,          
  citecolor   = blue             
}

\newcommand{\Rub}{\operatorname{Rub}}

\theoremstyle{definition}
\newtheorem{theorem}{Theorem}[section]

\newtheorem{lemma}[theorem]{Lemma}

\newtheorem{remark}[theorem]{Remark}

\newtheorem*{runningexample*}{Running example}

\newtheorem*{aside*}{Aside}

\newtheorem{definition}[theorem]{Definition}
\newtheorem{example}[theorem]{Example}

\newtheorem{proposition-definition}[theorem]{Proposition-Definition}

\DeclareMathOperator{\Hom}{Hom}

\DeclareMathOperator{\Bl}{Bl}

\DeclareMathOperator{\gp}{gp}

\newcommand{\RR}{\mathbb{R}}

\newcommand{\Ycal}{\mathcal{Y}}
\newcommand{\Zcal}{\mathcal{Z}}
\newcommand{\Exp}{\operatorname{Exp}}

\newcommand{\f}{\mathrm{f}}
\newcommand{\p}{\mathrm{p}}
\newcommand{\rtrop}{\mathrm{r}}

\newcommand{\Gm}{\mathbb{G}_{\operatorname{m}}}

\newcommand{\ol}[1]{\overline{#1}}

\newcommand{\bcd}{\begin{center}\begin{tikzcd}}
\newcommand{\ecd}{\end{tikzcd}\end{center}}

\newcommand{\Aaff}{\mathbb{A}}

\newcommand{\T}{\mathsf{T}}

\newcommand{\PP}{\mathbb{P}}
\newcommand{\OO}{\mathcal{O}}

\newcommand{\Speck}{\operatorname{Spec}\kfield}
\newcommand{\kfield}{\Bbbk}

\newcommand{\Acal}{\mathcal{A}}

\newcommand{\Xcal}{\mathcal{X}}
\newcommand{\Bcal}{\mathcal{B}}

\newcommand{\Xfrak}{\mathfrak{X}}

\newcommand{\Kup}{\mathsf{K}}

\newcommand{\acts}{\curvearrowright}

\newcommand{\edit}[1]{\color{red} {#1} \color{black}}

\NewDocumentCommand{\compatibilitydatum}{m m m m m m O{} O{} O{}}{
\begin{equation*} \begin{tikzcd}[ampersand replacement=\&]
  \: \arrow{r} \& {#1} \arrow{r} \arrow{d}{#7} \& {#2} \arrow{r} \arrow{d}{#8} \& {#3} \arrow{r}{[1]} \arrow{d}{#9} \& \: \\
  \: \arrow{r} \& {#4} \arrow{r} \& {#5} \arrow{r} \& {#6} \arrow{r} \& \:
\end{tikzcd} \end{equation*}}

\NewDocumentCommand{\commutingsquare}{m m m m o O{} O{} O{} O{}}{
\begin{equation}\begin{tikzcd}[ampersand replacement=\&] \label{#5}
  #1 \arrow{r}{#6} \arrow{d}{#7} \& #2 \arrow{d}{#8} \\
  #3 \arrow{r}{#9} \& #4
\end{tikzcd}\IfValueTF{#5}{\label{#5}}{} \end{equation}}

\NewDocumentCommand{\cartesiansquare}{m m m m O{} O{} O{} O{}}{
\begin{equation*}\begin{tikzcd}[ampersand replacement=\&]
  #1 \arrow{r}{#5} \arrow{d}{#6} \arrow[dr, phantom, "\square"] \& #2 \arrow{d}{#7} \\
  #3 \arrow{r}{#8} \& #4
\end{tikzcd} \end{equation*}}

\NewDocumentCommand{\cartesiansquarelabel}{m m m m m O{} O{} O{} O{}}{
\begin{tikzcd}[ampersand replacement=\&]
  #1 \arrow{r}{#6} \arrow{d}{#7} \arrow[dr, phantom, "\square"] \& #2 \arrow{d}{#8} \\
  #3 \arrow{r}{#9} \& #4
\end{tikzcd}\IfValueTF{#5}{\label{#5}}{}
}

\NewDocumentCommand{\triangleofspaces}{m m m O{} O{} O{}}{
\begin{tikzcd} [ampersand replacement=\&]
#1 \arrow{r}{#4} \arrow[bend right]{rr}{#5} \& #2 \arrow{r}{#6} \& #3
\end{tikzcd}}

\newcommand\blfootnote[1]{%
  \begingroup
  \renewcommand\thefootnote{}\footnote{#1}%
  \addtocounter{footnote}{-1}%
  \endgroup
}

\begin{document}
 
\title{Rubber tori in the boundary of expanded stable maps}
\author{Francesca Carocci and Navid Nabijou}

\begin{abstract} We investigate torus actions on logarithmic expansions in the context of enumerative geometry. Our main result is an intrinsic and coordinate-free description of the higher-rank rubber torus appearing in the boundary of the space of expanded stable maps. The rubber torus is constructed canonically from the tropical moduli space, and its action on each stratum of the expanded target is encoded in a linear tropical position map. The presence of $2$-morphisms in the universal target forces expanded stable maps differing by the rubber action to be identified. This provides the first step towards a recursive description of the boundary of the expanded moduli space, with future applications including localisation and rubber calculus.\end{abstract}
\vspace{-1cm}

\maketitle
\tableofcontents
\vspace{-1cm}

\section*{Introduction}

\noindent \blfootnote{MSC 2020: 14N35, 14H10, 14D06, 14D23 (Primary).} The development of logarithmic Gromov--Witten theory over the past decade has seen it grow to occupy a key position within contemporary enumerative geometry \cite{GrossSiebertLog,ChenLog,AbramovichChenLog,KimLog,RangExpansions}. This growth has been aided crucially by a corpus of techniques unique to the logarithmic theory, chief among them correspondence theorems involving tropical curve counts and scattering diagrams \cite{MikhalkinTropicalCorrespondence,NishinouSiebert,MandelRuddat,GPS,ArguzGross,Graefnitz,BousseauTakahashi,BBvG2}.

On the other hand, the powerful techniques of classical Gromov--Witten theory, such as torus localisation and rubber calculus, have yet to be successfully implemented in the logarithmic setting. This is because of one fundamental obstacle: the recursive structure of the boundary of the moduli space of logarithmic stable maps is intricate, and still not fully understood (though recent progress has been made in the unexpanded setting \cite[Theorem~5.6]{PuncturedMaps}, see the discussion of punctured maps below).

We provide the first step towards such a recursive description in the setting of expanded stable maps. We give an intrinsic tropical characterisation of the generalised rubber torus and its action (\S \ref{sec: rubber action}), together with a careful account of its relationship to the $2$-morphisms in the universal expanded target (\S \ref{sec: 2-morphisms} and \ref{sec: boundary of moduli}).

The rubber formalism is useful for relating logarithmic and absolute invariants. For instance, it is natural to conjecture that the logarithmic Gromov--Witten theory of a normal crossings pair is determined by the absolute Gromov--Witten theories of all its strata, generalising \cite[Theorem~2]{MaulikPandharipande}. Rubber calculations, combined with degeneration techniques, provide an avenue for approaching this question. 

\subsection*{Rubber tori and tropical position maps} In the expanded formalism \cite{RangExpansions,MR20} boundary strata parametrise transverse stable maps to a fixed target expansion. For smooth pairs, it is known that two such maps should be identified if they differ by the action of the rubber tori on the higher levels of the expanded target \cite{Li1}. Until now, the question of what the corresponding identification should be for normal crossings divisors has not been explored in detail. We answer this question, and hence provide the first detailed examination of the boundary of the  moduli space. The answer we obtain is satisfyingly tropical. 

Take a normal crossings pair $(X|D)$. A logarithmic target expansion corresponds to a combinatorial type of polyhedral subdivision of the tropicalisation $\Sigma=\Sigma(X|D)$. The irreducible components $Y_v$ of the expanded target are indexed by the vertices $v$ of this polyhedral subdivision.

Given a combinatorial type of polyhedral subdivision there is an associated tropical moduli space $\tau$ parametrising its edge lengths. Each vertex $v$ is contained in a unique minimal cone $\sigma_v \leq \Sigma$, and there is a linear map of cones called the \textbf{tropical position map}
\[ \varphi_v \colon \tau \to \sigma_v \]
which records the position of the vertex $v$ in terms of the tropical edge lengths. Our main result can be summarised as follows (full details are given in \S \ref{sec: rubber action}).

\begin{customthm}{A}[\emph{Theorem~\ref{prop: rubber action general}} \emph{for busy people}] The rubber torus is the torus associated to the tropical moduli space:
\[ T_\tau = \tau \otimes \Gm. \]
The action of the rubber torus on the irreducible component $Y_v$ of the expanded target is given by composing the homomorphism of tori associated to the tropical position map
\[ \varphi_v \otimes \Gm \colon T_\tau \to T_{\sigma_v}\]
with the natural action $T_{\sigma_v} \acts Y_v$.
\end{customthm}
Besides its importance for enumerative geometry, this result gives the first general characterisation of automorphisms of logarithmic expansions.

\subsection*{Example} Let $(X|D)$ be a normal crossings pair with $D=D_1+D_2$ consisting of two smooth components with connected nonempty intersection. Consider the following polyhedral subdivision of the tropicalisation:

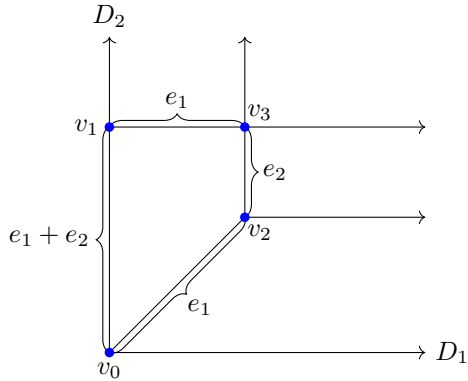
\begin{center}
\begin{tikzpicture}[scale=0.6]
\coordinate (OP) at (0,0,0);
\coordinate (BP) at (7,0,0);
   \coordinate (CP) at (0,7,0);
 \coordinate (v1P) at (0,5,0);
      \coordinate (v2P) at (3,3,0);
            \coordinate (v3P) at (3,5,0);
      \coordinate (vl2) at (7,3,0);
      \coordinate (vl31) at (7,5,0);
    \coordinate (vl32) at (3,7,0);
    
\begin{scope}[every coordinate/.style={shift={(11,0,0)}}]


\draw [decorate,decoration={brace,amplitude=5pt,mirror},xshift=0.4pt,yshift=0.4pt] ([c]OP)--([c]v2P) node[black,midway,xshift=0.3cm,yshift=-0.3cm] {\small{$e_1$}};
\draw [decorate,decoration={brace,amplitude=5pt,mirror},xshift=0.4pt,yshift=-0.4pt] ([c]v2P)--([c]v3P) node[black,midway,xshift=0.4cm,yshift=0 cm] {\small{$e_2$}};
\draw [decorate,decoration={brace,amplitude=5pt,mirror},xshift=0.4pt,yshift=0.4pt] ([c]v3P)--([c]v1P) node[black,midway,xshift=0cm,yshift=.35cm] {\small{$e_1$}};
\draw [decorate,decoration={brace,amplitude=5pt,mirror},xshift=0.4pt,yshift=-0.4pt] ([c]v1P)--([c]OP) node[black,midway,xshift=-0.8cm,yshift=0.05 cm] {\small{$e_1+e_2$}};

\draw[->] ([c]OP)--([c]BP);
\draw [->]([c]OP)--([c]CP);
\draw ([c]OP)--([c]v2P);
\draw ([c]v2P)--([c]v3P);
\draw ([c]v1P)--([c]v3P);
\draw [->]([c]v2P)--([c]vl2);
\draw [->]([c]v3P)--([c]vl31);
\draw [->]([c]v3P)--([c]vl32);

   \fill[blue]([c]OP) circle (3pt);
       \fill[blue] ([c]v1P) circle (3pt);
       \fill[blue] ([c]v2P) circle (3pt);
          \fill[blue] ([c]v3P) circle (3pt);
       
       \node at ([c]OP) [below] {\small{$v_0$}};	
     \node at ([c]v1P) [left] {\small{$v_1$}};
     \node at ([c]v2P) [xshift=.2cm, yshift=-0.2cm] {\small{$v_2$}};
          \node at ([c]v3P) [xshift=.2cm, yshift=0.2cm] {\small{$v_3$}};

      \node at ([c]BP) [right] {\small{$D_1$}};
          \node at ([c]CP) [above] {\small{$D_2$}};

\end{scope}
\end{tikzpicture}
\end{center}
This gives rise to a target expansion consisting of four irreducible components, indexed by the vertices: $Y_{v_0}$ is the blowup of $X$ at $D_1 \cap D_2$; $Y_{v_1}$ is a $\PP^1$ bundle over $D_2$; $Y_{v_2}$ and $Y_{v_3}$ are respectively $\PP^2$ and $\PP^1 \times \PP^1$ bundles over $D_1 \cap D_2$. The minimal cones of $\Sigma$ containing the vertices $v_i$ are
\[ \sigma_{v_0}=0, \quad \sigma_{v_1} = \RR_{\geq 0} \langle D_2 \rangle, \quad \sigma_{v_2} = \RR^2_{\geq 0} \langle D_1,D_2 \rangle, \quad \sigma_{v_3} = \RR^2_{\geq 0} \langle D_1, D_2 \rangle, \]
and each $T_{\sigma_{v_i}}=\sigma_{v_i} \otimes \Gm$ acts on the toric variety bundle $Y_{v_i}$ as the dense torus of the fibres. The tropical moduli space associated to the polyhedral subdivision is
\[ \tau=\RR^2_{\geq 0}\langle e_1,e_2 \rangle \]
and the tropical position maps $\varphi_{v_i} \colon \tau \to \sigma_{v_i}$ are given as follows:
\begin{align*} \varphi_{v_0}(e_1,e_2) & = 0, \\
\varphi_{v_1}(e_1,e_2) & = e_1 + e_2,  \\
\varphi_{v_2}(e_1,e_2) & = (e_1,e_1), \\
\varphi_{v_3}(e_1,e_2) & = (e_1,e_1+e_2). \end{align*}
The rubber torus is $T_\tau=\Gm^2$ and each tropical position map induces a homomorphism of tori
\[ \varphi_{v_i} \otimes \Gm \colon T_\tau \to T_{\sigma_{v_i}} \]
which together with the action of $T_{\sigma_{v_i}}$ on $Y_{v_i}$ determines the action of the rubber torus on $Y_{v_i}$. For example, the action on $Y_{v_2}$ has weights $(\lambda_1,\lambda_1)$ in the affine coordinates of the fibres of the $\PP^2$ bundle.

Many more examples are presented in \S \ref{sec: examples}.

\subsection*{Moving the join divisors} The example above illustrates an important new phenomenon which appears as soon as we move beyond the case of smooth pairs. Rather than every component of the expanded target carrying its own associated torus, there is instead a \emph{single} rubber torus $T_\tau$ which acts on the entire expanded target simultaneously. This torus cannot, in general, be split into tori acting on the individual components, because of the following basic fact:
\begin{center}
\fbox{\parbox{0.8\textwidth}{\center{The rubber torus can act nontrivially on a divisor joining two components of the expanded target.}}}
\end{center}
\noindent In the previous example, the rubber torus $T_\tau$ acts nontrivially on both $Y_{v_1} \cap Y_{v_3}$ and $Y_{v_2} \cap Y_{v_3}$.

This fact reflects the rich geometry of higher-rank target expansions, and complicates the search for a simple recursive description of the boundary. We discuss these issues, and possible solutions, in \S \ref{sec: obstacles for recursive description}.

\subsection*{$2$-morphisms} Having defined the rubber torus and its action (\S \ref{sec: rubber action}), we next give a conceptual explanation as to why it appears in the boundary of the moduli space of expanded stable maps (\S \ref{sec: 2-morphisms} and \ref{sec: boundary of moduli}). The key point is that the universal expanded target
\[ \Xfrak \to \Exp \]
is a morphism between two Artin stacks \cite{MR20}. As such, the $2$-morphisms in $\Xfrak$ and $\Exp$ must be incorporated when formulating the notion of isomorphism between expanded stable maps. Once this is done correctly, the rubber action arises automatically.

The upshot is Theorem~\ref{thm: description of locally closed stratum}, which describes the locally-closed strata of the moduli space in terms of stable maps to a fixed target expansion, identified up to the rubber action. In \S \ref{sec: rubber spaces}, we use the same perspective to define expanded stable maps to higher-rank rubber targets.

\subsection*{Future directions} As already noted, our results provide the first steps towards a recursive description of the boundary of the moduli space of expanded stable maps. The remaining obstacles are discussed in \S \ref{sec: obstacles for recursive description}. It seems likely that a naive splitting formalism, into moduli spaces associated to each irreducible component of the target, does not exist. Instead, we speculate that a combination of coarser splitting statements (e.g. into moduli spaces associated to each maximal cone of the unexpanded tropical target) and rigidification techniques will be required to achieve calculations.

We focus on logarithmic Gromov--Witten theory, but our arguments apply \emph{mutatis mutandis} to the boundary of the logarithmic Donaldson--Thomas moduli space \cite{MR20}, see Remark~\ref{rmk: can also do DT}.

In \cite{CarocciNabijou2} we describe the global geometry of the irreducible components $Y_v$ of a logarithmic expansion. This turns out to be more subtle than it initially appears.

\subsection*{Punctured maps} The theory of punctured stable maps \cite{PuncturedMaps} provides an alternative avenue for exploring recursive descriptions of the boundary in the unexpanded setting. In our experience the unexpanded and expanded theories complement each other well, with their differences usually making one or the other better suited to a particular purpose. It would be worthwhile to compare and contrast the two recursive formalisms.

\subsection*{Acknowledgements} We thank Dhruv~Ranganathan for numerous helpful discussions. We also benefited from stimulating conversations with Luca~Battistella and Patrick~Kennedy-Hunt. We thank the anonymous referee for several valuable suggestions.

\subsection*{Funding} F.C. is supported by the EPFL Chair of Arithmetic Geometry (ARG). N.N. is supported by the Herchel Smith Fund. 

\subsection*{Notation and conventions} We work over an algebraically closed field $\kfield$ of characteristic zero, and assume familiarity with the basics of logarithmic and toroidal geometry.

A cone is by definition a pair $(\sigma, N_\sigma)$ where $N_\sigma$ is a lattice of finite rank and $\sigma \subseteq N_{\sigma}\otimes \RR$ is a convex rational polyhedral cone. We do not require $\sigma$ to be strictly convex. Morphisms of cones must respect the integral structure. A face morphism is an injective morphism of cones whose image is a face of the codomain. A diagram of cones and face morphisms is admissible if it contains all identity morphisms, and if every face of every cone belongs to the image of a face morphism. By a \textbf{cone stack} we mean the colimit of an admissible diagram. A cone stack is a \textbf{cone space} if the diagram contains at most one isomorphism between two distinct cones, and no automorphisms of a given cone besides the identity. A cone space is a \textbf{cone complex} if there is at most one morphism between any pair of cones \cite[\S 2.2]{CavalieriChanUlirschWise}.

We collect for the reader's convenience the principal characters of the story. These are described in detail in the text.

\tabulinesep=3pt
\begin{longtabu}{p{0.1\textwidth}  p{0.8\textwidth}}
$\tau$ & rational polyhedral cone; later a cone in the tropical moduli of expansions\\
$N_\tau$ & lattice associated to $\tau$\\
$Q_\tau$ & dual monoid $Q_\tau = \tau^\vee \cap N_\tau^\vee$\\
$U_\tau$ & affine toric variety $U_\tau = \Speck[Q_\tau]$\\
$T_\tau$ & dense torus in $U_\tau$, i.e. the rubber torus: $T_\tau = \Speck[Q_\tau^{\gp}]= \tau \otimes \Gm$\\
$\Acal_\tau$ & Artin cone $\Acal_\tau = [U_\tau/T_\tau]$\\
$(X|D)$ & normal crossings pair \\
$\Sigma$ & tropicalisation of $(X|D)$\\
$\Upsilon$ & tropical expansion of $\Sigma$ over $\tau$; conical subdivision of $\Sigma \times \tau$ \\
$\Xcal_{\Upsilon}$ & logarithmic expansion of $X$ over $U_\tau$; logarithmic modification of $X \times U_\tau$\\
$Y_\Upsilon$ & central fibre $\Xcal_{\Upsilon}|_0$ over $0 \in U_\tau$\\
$Y_v$ & component of $Y_\Upsilon$ indexed by vertex $v$ of associated polyhedral subdivision.\\
$\T$ & tropical moduli space of target expansions, constructed in \cite{MR20}\\
$\Exp$ & moduli stack of target expansions; Artin fan associated to $\T$\\
$\Xcal_{\tau}$ & versal target over $U_\tau$ \\
$\Xfrak_\tau$ & universal target over $\Acal_\tau$, quotient of versal target by the rubber action\\
$\Xfrak$ & universal target over $\Exp$, obtained by gluing the $\Xfrak_\tau$
\end{longtabu}

\section{Rubber actions on logarithmic expansions} \label{sec: rubber action}
\noindent We begin with a direct and coordinate-free description of the rubber torus, and its action on the expanded target, in terms of tropical moduli and position maps. The key points are Definition~\ref{def: position map} and Theorem~\ref{prop: rubber action general}. The starting data is:
\begin{enumerate}
\item $(X|D)$ a normal crossings pair;
\item $\tau$ a rational polyhedral cone, of maximal dimension in $N_\tau$.
\end{enumerate}
We will study families of logarithmic expansions of $(X|D)$ parametrised by $\tau$. In the setting of expanded stable maps, $\tau \leq \T$ will be a cone in a tropical moduli space of target expansions.

While working through this section, the reader may find it helpful to refer to the examples presented in \S \ref{sec: examples}.

\subsection{Subdivision preliminaries} \label{sec: subdivisions} Cone complexes, and subdivisions thereof, play a central role in the theory of expansions. We recall the following definition, which deviates from standard usage.
\begin{definition}[\hspace{0.1pt}{\cite[Definition 1.3.1]{MR20}}] \label{def: subdivision} Let $\Delta$ be a cone complex. A (conical) \textbf{subdivision} of $\Delta$ is a morphism of cone complexes
\begin{equation*} \Delta^\dag \to \Delta \end{equation*}
which is injective on supports, and such that the image of the set of integral points of every cone in $\Delta^\dag$ is a saturated submonoid of the set of integral points of a cone in $\Delta$.\end{definition}
This definition allows for morphisms $\Delta^\dag \to \Delta$ which are injective but not surjective on supports. This additional flexibility is useful for applications in enumerative geometry. Subdivisions which are bijective on supports, i.e. which are subdivisions in the usual sense, will be referred to as \textbf{complete}. This special class enjoys no real advantages: as we shall see, complete and non-complete subdivisions can be handled uniformly.

Toroidal semistable reduction \cite{AbramovichKaru} can be used to factor every subdivision as:
\begin{equation*} \Delta^\dag \to \Delta^\ddag \to \Delta \end{equation*}
where $\Delta^\ddag \to \Delta$ is a complete subdivision and $\Delta^\dag \to \Delta^\ddag$ is the inclusion of a subcomplex. There is typically no preferred or minimal choice for $\Delta^\ddag \to \Delta$.

 Suppose that $\Delta$ is the tropicalisation of a logarithmic scheme $X$ \cite{AbramovichEtAlSkeletons}. Then a subdivision $\Delta^\dag \to \Delta$ induces a morphism of logarithmic schemes
 \begin{equation*} X^\dag \to X \end{equation*}
 which is logarithmically \'etale, but is not proper or surjective unless the subdivision is complete. We will again deviate from standard terminology, and refer to such morphisms as \textbf{logarithmic modifications}. By the previous paragraph, $X^\dag$ may be realised (non-uniquely) as an open subset of a logarithmic modification in the usual sense.

\subsection{Tropical expansions} \label{sec: tropical expansions general case} Let $(X|D)$ be a normal crossings pair with tropicalisation $\Sigma=\Sigma(X|D)$. We consider tropical expansions of $\Sigma$ parametrised by an arbitrary base cone $\tau$. These have been studied by many authors, in various guises and levels of generality \cite{KKMS,MumfordAbelian,NishinouSiebert,MandelRuddat,FosterRanganathan,MR20}.
\begin{definition}\label{def: tropical expansion general} A \textbf{tropical expansion} of $\Sigma$ over $\tau$ consists of a conical subdivision of cone complexes
\begin{equation*} \Upsilon = (\Sigma \times \tau)^\dag \to \Sigma \times \tau \end{equation*}
such that for every cone $\omega \leq \Upsilon$ the projection $\p \colon \Upsilon \to \tau$:
\begin{itemize}
\item maps $\omega$ surjectively onto a face of $\tau$ ($\p$ is \emph{integral} or \emph{combinatorially flat});
\item maps $N_{\omega}$ onto the lattice of its image face ($\p$ is \emph{saturated} or \emph{combinatorially reduced}).
\end{itemize}
\end{definition}
\noindent Following Definition~\ref{def: subdivision}, the inclusion of supports $|\Upsilon| \subseteq |\Sigma \times \tau|$ is injective but not necessarily surjective. Tropical expansions with $|\Upsilon|=|\Sigma \times \tau|$ are referred to as \textbf{complete}. The following discussion applies uniformly to complete and non-complete expansions.

Fix a tropical expansion $\p \colon \Upsilon \to \tau$. For each point $\f \in |\tau|$ the fibre
\begin{equation*} \p^{-1}(\f) \end{equation*}
is a polyhedral complex. It is a polyhedral subdivision of $\Sigma$ in the generalised sense, i.e. there is a morphism of polyhedral complexes $\p^{-1}(\f) \to \Sigma$ which is injective on supports. 

We view $\p$ as a family of polyhedral subdivisions of $\Sigma$, parametrised by $\f \in |\tau|$. The \textbf{combinatorial type} of $\p^{-1}(\f)$ is the data which associates to each polyhedron $P \leq \p^{-1}(\f)$ the minimal cone $\sigma_P \leq \Sigma$ containing it, and associates to each oriented edge $E \leq \p^{-1}(\f)$ its integral slope $m_E \in N_{\sigma_E}$. This data is constant over the relative interior of each face of $\tau$. 

The finite edges of $\p^{-1}(\f)$ are metrised by the choice of point in the base. Specialising to a smaller face of $\tau$ has the effect of setting some of these edge lengths to zero, collapsing the polyhedral subdivision to a simpler combinatorial type. In the limit $\f=0$ we obtain a cone complex $\Sigma^\dag=\p^{-1}(0)$ called the \textbf{asymptotic cone complex}. It is a conical subdivision of $\Sigma$. 

Let $U_\tau$ be the affine toric variety associated to $\tau$. The product $\Sigma\times \tau$ is the tropicalisation of the product $X \times U_\tau$ and so the subdivision $\Upsilon$ produces a logarithmic modification:
\begin{equation*} \Xcal_{\Upsilon} \to X \times U_\tau.\end{equation*}
 This morphism is logarithmically \'etale, and the second projection of $X \times U_\tau$ is logarithmically smooth. Therefore the composite
 \begin{equation*} \pi \colon \Xcal_{\Upsilon} \to U_\tau \end{equation*}
is logarithmically smooth. Since its tropicalisation $\p$ is integral and saturated, it follows that $\pi$ is flat with reduced fibres \cite[Lemmas 4.1 and 5.2]{AbramovichKaru}. We refer to this as a \textbf{logarithmic expansion} of $X$ parametrised by $U_\tau$. The first projection
\begin{equation*} \rho \colon \Xcal_{\Upsilon} \to X \end{equation*}
collapses each fibre of $\pi$ onto the unexpanded target $X$. 

The general fibre of $\pi$ is the logarithmic modification $X^\dag \to X$ associated to the asymptotic cone complex $\Sigma^\dag \to \Sigma$. The central fibre over the torus-fixed point $0 \in U_\tau$ is denoted:
\begin{equation*} Y_\Upsilon = \Xcal_{\Upsilon}|_0.\end{equation*}
The irreducible components of $Y_\Upsilon$ are indexed by the vertices of the polyhedral subdivision $\p^{-1}(\f)\to\Sigma$, for $\f \in |\tau|$ any interior point. Let $V$ denote the set of such vertices and for $v \in V$ let $Y_v$ denote the corresponding component, so that:
\begin{equation*} Y_\Upsilon = \bigcup_{v \in V} Y_v.\end{equation*}
For $v \in V$ let $\sigma_v \leq \Sigma$ be the minimal cone containing the vertex $v$ and let $X_v = X_{\sigma_v} \subseteq X$ denote the corresponding stratum. The fibrewise collapsing morphism $\rho \colon \Xcal_{\Upsilon} \to X$ restricts to a morphism $Y_v \to X_v$.

In certain situations $Y_v \to X_v$ is a toric variety bundle. In general, however, this is not the case. The morphism can fail to be flat, or can be flat but with reducible fibres. However, if we restrict to the interiors of $Y_v$ and $X_v$ we obtain a principal torus bundle:
\begin{equation*} Y_v^\circ \to X_v^\circ. \end{equation*}
We will describe the rubber action on $Y_v^\circ$ and, more generally, on every locally-closed stratum of $Y_\Upsilon$. This is sufficient for applications to enumerative geometry. The global geometry of the irreducible components $Y_v$, including a combinatorial recipe to construct $Y_v$ from the base $X_v$, is treated in the companion paper \cite{CarocciNabijou2}.

\subsection{Rubber action (toric case)} \label{sec: rubber action toric case} We start with the toric case, which already contains all the essential ideas. Assume that $X$ is a toric variety with $D=\partial X$ the toric boundary. The tropicalisation $\Sigma=\Sigma_X$ is simply the fan of $X$.

The dense torus $T_\Sigma \subseteq X$ will not exist in the general setting of toroidal embeddings (Section~\ref{sec: rubber action general case}), and plays only an auxiliary role in the forthcoming discussion. The dense torus $T_\tau$, on the other hand, always exists and is of fundamental importance:
\begin{definition}\label{def: rubber torus} The \textbf{rubber torus} is the dense torus of the base:
\begin{equation} T_\tau = \tau \otimes \Gm \subseteq U_\tau.\end{equation}
The \textbf{rubber action} is the natural action
\begin{equation*} T_\tau \acts \Xcal_{\Upsilon} \end{equation*} 
induced by the splitting $T_\Upsilon = T_\Sigma \times T_\tau$. The morphisms $\pi$ and $\rho$ are equivariant with respect to this action. In particular, the rubber action restricts to an action on the central fibre:
\begin{equation} T_\tau \acts Y_\Upsilon. \end{equation}
\end{definition}
We now give a tropical description of the rubber action on each component $Y_v$ of $Y_\Upsilon$. The key idea is contained in the following definition.
\begin{definition}\label{def: position map}
For each $v \in V$ let $\sigma_v \leq \Sigma$ denote the unique minimal cone containing $v$. The \textbf{tropical position map} is the linear map
\begin{equation} \varphi_v \colon \tau \to \sigma_v \end{equation}
which for each $\f \in |\tau|$ records the position $\varphi_v(\f) \in \sigma_v$ of the vertex $v$ of the polyhedral complex $\p^{-1}(\f)$. Formally: the vertex $v$ corresponds to a cone $\omega_v \leq \Upsilon$ which is mapped isomorphically onto $\tau$ via $\p$, and the tropical position map is defined as
\begin{equation}\label{eqn: formula for position map} \varphi_v(\f) = \rtrop(\omega_v \cap \p^{-1}(\f)) \in \sigma_v \end{equation}
for all $\f \in |\tau|$, where $\rtrop \colon \omega_v \to \sigma_v$ is the tropicalisation of $\rho$.
\end{definition}

\begin{theorem}\label{prop: rubber action toric}
Let $T_v \subseteq Y_v$ be the dense torus. Then there is a natural inclusion $T_{\sigma_v} \hookrightarrow T_v$ and the rubber action $T_\tau \acts Y_v$ is determined by the tropical position map, as the composite:
\begin{equation} T_\tau \xrightarrow{\varphi_v \otimes \Gm} T_{\sigma_v} \hookrightarrow T_v.\end{equation} \end{theorem}
\begin{proof} The component $Y_v$ is a toric stratum in $\Xcal_{\Upsilon}$. The corresponding cone $\omega_v$ is given set-theoretically by:
\begin{equation*} \omega_v = \{ (\varphi_v(\f),\f) \colon \f \in \tau \} \subseteq |\sigma_v \times \tau| \subseteq |\Sigma \times \tau|.\end{equation*}
This is simply the image of the section $\varphi_v \times \operatorname{Id}_\tau$ of $\p|_{\omega_v}$ discussed in Definition~\ref{def: position map}. Recall that for any cone or fan $\rho$, we let $N_\rho$ denote the associated lattice. The dense torus $T_v \subseteq Y_v$ has cocharacter lattice
\begin{equation*} N_v = N_\Upsilon / N_{\omega_v} = (N_{\Sigma} \times N_\tau)/N_{\omega_v}\end{equation*}
and the action of the rubber torus is given by the composition:
\begin{equation}\label{eqn: rubber action 1} N_{\tau} \hookrightarrow N_{\Sigma} \times N_\tau \rightarrow N_{v}.\end{equation}
On the other hand, the morphism
\begin{align*}
 \psi \colon N_{\sigma_v} \times N_\tau & \to N_{\sigma_v}\\
(\mathrm{g},\mathrm{h}) & \mapsto \varphi_v(\mathrm{h}) - \mathrm{g}	
\end{align*}
is surjective with kernel $N_{\omega_v}$ and hence gives rise to an identification:
\begin{equation*} (N_{\sigma_v} \times N_\tau )/N_{\omega_v} = N_{\sigma_v}.\end{equation*}
Since $N_{\sigma_v} \subseteq N_\Sigma$ this identification provides a natural inclusion $\iota \colon N_{\sigma_v} \hookrightarrow N_v$. This allows us to factor the composite \eqref{eqn: rubber action 1} as
\begin{equation*} N_{\tau} \overset{\jmath}{\rightarrow} N_{\sigma_v} \times N_\tau \overset{\psi}{\rightarrow} N_{\sigma_v} \overset{\iota}{\rightarrow} N_v \end{equation*}
and since $\psi\circ \jmath = \varphi_v$ we arrive at the desired description. \end{proof}

\subsection{Rubber action on locally-closed strata (toric case)} \label{sec: locally closed strata toric case} The above description specialises to a description of the rubber action on each locally closed stratum of the central fibre $Y_{\Upsilon}$. This will be useful for describing the rubber action in the general case, so we provide the details here.

The locally closed strata of the central fibre $Y_\Upsilon \subseteq \Xcal_{\Upsilon}$ are indexed by cones $\omega \leq \Upsilon$ such that $\p(\omega)=\tau$. The poset of such cones is isomorphic to the poset of polyhedra in the polyhedral complex $\p^{-1}(\mathrm{f})$, where $\mathrm{f} \in |\tau|$ is any interior point.

Consider such a polyhedron $P$ with associated cone $\omega_P \leq \Upsilon$ and let $Y_P^\circ \subseteq Y_\Upsilon$ denote the corresponding locally-closed stratum. Let $\sigma_P \leq \Sigma$ be the unique minimal cone containing the polyhedron $P$ and let $X_P^\circ \subseteq X$ denote the corresponding locally-closed stratum.  By definition, $\sigma_P$ is the minimal cone through which the map $\omega_P \to \Sigma$ factors, and so $\rho \colon \Xcal_{\Upsilon} \to X$ restricts to a morphism:
\begin{equation*} Y_P^\circ \to X_P^\circ.\end{equation*}
This is a principal torus bundle. We will now describe its structure group, and use this to describe the rubber torus action on its fibres.

Since $\Upsilon \to \Sigma \times \tau$ restricts to $\omega_P \to \sigma_P \times \tau$ there is an inclusion $N_{\omega_P} \hookrightarrow N_{\sigma_P} \times N_\tau$. The morphism $N_{\omega_P} \to N_\tau$ is surjective since $\p \colon \Upsilon \to \tau$ is integral and saturated. We let $K_P$ denote the kernel:
\begin{equation*} K_P = \ker(N_{\omega_P} \to N_\tau).\end{equation*}
 There is a natural inclusion $K_P \hookrightarrow N_{\sigma_P}$ and we denote the quotient by
\begin{equation*} N_{\theta_P} \colonequals N_{\sigma_P}/K_P \end{equation*}
with associated cone:
\begin{equation*} \theta_P \colonequals \operatorname{im}(\sigma_P) \subseteq N_{\theta_P} \otimes \RR.\end{equation*}
We note that $\theta_P$ may not be strictly convex. In polyhedral terms, $K_P \subseteq N_{\sigma_P}$ is the lattice spanned by directions in $P$; working locally around a generic point of $P$ and declaring this point to be the origin, the inclusion $P \subseteq \sigma_P$ is a conical subset, with $K_P\subseteq N_{\sigma_P}$ the associated lattice. Note that:
\begin{equation*} \dim \theta_P = \dim \sigma_P - \dim P.\end{equation*}
We also have:
\begin{equation*} (N_{\sigma_P}\times N_\tau)/N_{\omega_P} = N_{\sigma_P}/K_P=N_{\theta_P}.\end{equation*}

\begin{theorem} \label{prop: rubber action toric locally closed} The morphism 
\begin{equation}\label{eqn: toric principal bundle} Y_P^\circ \to X_{P}^\circ \end{equation}
is a principal torus bundle, with structure group $T_{\theta_P}=\theta_P \otimes \Gm$. There is a well-defined linear tropical position map
\begin{equation*} \varphi_P \colon \tau \to \theta_P \end{equation*}
such that the rubber action on the fibres of \eqref{eqn: toric principal bundle} is given by $\varphi_P \otimes \Gm$. 
\end{theorem}

\begin{proof} The torus $Y_P^\circ$ has cocharacter lattice $(N_\Sigma \times N_\tau)/N_{\omega_P}$ while the torus $X_{P}^\circ$ has cocharacter lattice $N_\Sigma/N_{\sigma_P}$. It follows that the subtorus of $Y_P^\circ$ which preserves the fibres has cocharacter lattice:
\begin{equation*} (N_{\sigma_P}\times N_\tau)/N_{\omega_P} =N_{\theta_P}.\end{equation*}
This proves the first part. For the second part, we consider the composite
\begin{equation*} N_\tau \to N_{\sigma_P} \times N_\tau \to (N_{\sigma_P} \times N_\tau)/N_{\omega_P} = N_{\theta_P}\end{equation*}
which by definition gives the rubber torus action. This clearly maps $\tau$ to $\theta_P$ and so we obtain the desired tropical position map $\varphi_P$. The description of $\theta_P$ as the quotient by directions appearing in $P$ shows that $\varphi_P$ indeed records the position of $P$, valued in the appropriate quotient vector space (see \S \ref{sec: examples} for examples). \end{proof}

\begin{remark} If $P=v$ is a vertex we have $K_P=0$ and $\theta_P=\sigma_P$. In this setting, the previous result refines Theorem~\ref{prop: rubber action toric} by identifying the torus $T_{\sigma_v}$ with the structure group of the principal bundle $Y_v^\circ \to X_v^\circ$. \end{remark}

\subsection{Rubber action (general case)} \label{sec: rubber action general case} Now consider an arbitrary normal crossings pair $(X|D)$ and a tropical expansion $\Upsilon$ of $\Sigma=\Sigma(X|D)$ with base cone $\tau$. The \textbf{rubber torus} is defined as in the toric case:
\begin{equation*} T_\tau = \tau \otimes \Gm.\end{equation*}
As before, for each locally-closed stratum $Y_P^\circ \subseteq Y_\Upsilon$ indexed by a polyhedron $P$, there is a well-defined quotient cone $\theta_P$ of $\sigma_P$ and a tropical position map $\varphi_P \colon \tau \to \theta_P$.

\begin{theorem}\label{prop: rubber action general} There is a canonical global rubber action
$$T_\tau \acts \Xcal_{\Upsilon}$$
which is equivariant with respect to both $\pi \colon \Xcal_{\Upsilon} \to U_\tau$ and $\rho \colon \Xcal_{\Upsilon} \to X$. For each locally-closed stratum $Y_P^\circ$ of the central fibre, the structure group of the principal bundle
\begin{equation*} \rho \colon Y_P^\circ \to X_P^\circ \end{equation*}
is naturally identified with $T_{\theta_P}$, and the action of the rubber torus $T_\tau$ on the fibres of $\rho$ is given by the tropical position map:
\begin{equation*} \varphi_P \otimes \Gm \colon T_\tau \to T_{\theta_P}.\end{equation*}
\end{theorem}

\begin{proof} For each cone $\sigma \leq \Sigma$ there is an associated open set $V_\sigma \subseteq X$ defined as the union of strata indexed by faces of $\sigma$. This is an atomic neighbourhood for the logarithmic scheme $(X|D)$ in the weak sense, i.e. there is a strict smooth morphism $V_\sigma \to \Acal_\sigma=[U_\sigma/T_\sigma]$ (note that typically it is necessary to restrict to a smaller open set in order to obtain an atomic neighbourhood in the usual sense). The $V_\sigma$ cover $X$ and their inclusion poset is isomorphic to the face-inclusion poset of $\Sigma$. Moreover the cover is closed under intersections. This weak atomic covering induces a similar covering of $X \times U_\tau$. Locally with respect to this covering, $\Xcal_{\Upsilon}$ is obtained from $X\times U_\tau$ by pulling back the associated toric modification:
\bcd
\Xcal_{\Upsilon}|_{V_{\sigma}\times U_\tau} \ar[r] \ar[d] \ar[rd,phantom,"\square"] & V_\sigma \times U_\tau \ar[d] \\
(\Acal_\sigma \times U_\tau)^\dag \ar[r] & \Acal_\sigma \times U_\tau.
\ecd
We have $\Acal_\sigma \times U_\tau = [(U_\sigma \times U_\tau)/T_{\sigma}]$ and $(\Acal_\sigma \times U_\tau)^\dag = [(U_\sigma \times U_\tau)^\dag/T_\sigma]$. The action $T_\tau \acts (U_\sigma \times U_\tau)^\dag$ commutes with the action of $T_\sigma$ and hence descends to an action on the quotient. The action
\begin{equation*} T_\tau \acts \Xcal_{\Upsilon}|_{V_{\sigma}\times U_\tau} \end{equation*}
is then defined via the cartesian square above, acting on $(\Acal_\sigma \times U_\tau)^\dag$ as above and acting on $V_\sigma \times U_\tau$ and $\Acal_\sigma \times U_\tau$ through the second factor.

If $\rho$ is a face of $\sigma$ then $V_{\rho} \subseteq V_{\sigma}$ and so $\Xcal_{\Upsilon}|_{V_{\rho}\times U_\tau} \subseteq \Xcal_{\Upsilon}|_{V_{\sigma}\times U_\tau}$. It is easy to see that the action of $T_{\tau}$ on $\Xcal_{\Upsilon}|_{V_{\sigma}\times U_\tau}$ restricts to the action on $\Xcal_{\Upsilon}|_{V_{\rho}\times U_\tau}$. Therefore the actions glue to produce a global action on $\Xcal_{\Upsilon}$. Locally the maps $\pi$ and $\rho$ are obtained as the composites
\begin{align*} & \Xcal_{\Upsilon}|_{V_{\sigma}\times U_\tau} \to (\Acal_\sigma \times U_\tau)^\dag \to U_\tau \\
& \Xcal_{\Upsilon}|_{V_{\sigma}\times U_\tau} \to (V_\sigma \times U_\tau) \to V_\sigma	
\end{align*}
respectively. By construction both maps are equivariant.

It remains to describe the action on the locally-closed strata of the central fibre. Since the principal bundles $Y_P^\circ \to X_P^\circ$ are pulled back from their toric counterparts, this follows directly from the toric case (Theorem~\ref{prop: rubber action toric locally closed}).
 \end{proof}

\section{$2$-morphisms in the universal target} \label{sec: 2-morphisms}

\noindent The rubber action furnishes every logarithmic expansion with a natural symmetry group. We now explain the relationship between this action and the moduli space of expanded stable maps. The root of the interplay is the fact that the universal expanded target is naturally an Artin stack. From this the rubber action arises naturally, as a consequence of the $2$-morphisms.

\subsection{Expansions: tropical, universal and versal} \label{sec: constructing universal target} The moduli space of stable maps to expansions is constructed \cite{RangExpansions} as a logarithmic modification of the space of logarithmic stable maps to the unexpanded target \cite{GrossSiebertLog,ChenLog,AbramovichChenLog}.

The construction proceeds by applying toroidal semistable reduction \cite{AbramovichKaru,Molcho} to the universal family over the unexpanded moduli space \cite[\S\S 2--3]{RangExpansions}. This approach is geometric: it expresses the expanded moduli space as an explicit logarithmic modification of the unexpanded moduli space. A drawback is that the space constructed is not shown to represent a moduli functor of predeformable maps to expanded targets.

 To properly formulate such a moduli functor, it is necessary to construct a suitable universal expanded target, and its underlying moduli space of target expansions. This is achieved in \cite[\S 3]{MR20} as a key input to defining logarithmic Donaldson--Thomas theory. We summarise the procedure.

\subsubsection{Tropical} \label{sec: tropical target} Both the universal expanded target and its underlying moduli space of target expansions are constructed from the corresponding tropical spaces, i.e. cone complexes. The \textbf{tropical moduli space of target expansions} \cite[\S 3]{MR20}
\begin{equation*} \T = \T(X|D) \end{equation*}
parametrises embedded $1$-complexes in $\Sigma=\Sigma(X|D)$. The construction proceeds as follows. To every combinatorial type of $1$-complex $\mathsf{G}$, there is an associated cone $\mathbb{X}_{\mathsf{G}}$ parametrising its edge lengths. The idea is to glue the $\mathbb{X}_{\mathsf{G}}$ together, with gluing maps indexed by appropriate reduction operations on the set of combinatorial types $\mathsf{G}$.

These reductions are captured in the notion of \emph{surjection} \cite[Definition 3.3.4]{MR20}. This encompasses two distinct operations on combinatorial types: contraction of finite edges (smoothing nodes), and introduction of bivalent vertices when the images of two edge interiors collide (component bubbling, see \cite[Example~3.3.3]{MR20}).

Every surjection $\mathsf{G} \to \mathsf{H}$ induces an injective cone map $\mathbb{X}_{\mathsf{H}} \hookrightarrow \mathbb{X}_{\mathsf{G}}$, but in the case of component bubbling the image may not be a face. Consequently, it is necessary to compatibly subdivide each $\mathbb{X}_{\mathsf{G}}$ in order to ensure that every surjection of combinatorial types induces an inclusion of a subcomplex. The system of subdivisions is non-unique, though all choices will lead to virtually birational moduli spaces of expanded stable maps. Once a choice is made, the subdivisions $\mathbb{Y}_{\mathsf{G}}$ may be glued together to form $\T$.

An additional complication arises from automorphisms of combinatorial types. This necessitates taking automorphism-equivariant subdivisions. The resulting diagram of cones contains self-maps, but a final barycentric subdivision removes these \cite[\S 2.6]{ACP}, resulting in a cone complex instead of a cone stack. We emphasise that it is only the conical structure of $\T$ that depends on choices; its support is canonically defined, with each point $\f \in |\T|$ corresponding to a unique embedded $1$-complex $\mathsf{G}_{\f}$ in $\Sigma$.

It remains to construct the \textbf{universal tropical expansion} \cite[\S 3.5]{MR20}:
\begin{equation*} \Upsilon = (\Sigma \times \T)^\dagger \overset{\p}{\longrightarrow} \T. \end{equation*}
Each of the $\mathbb{Y}_{\mathsf{G}}$ constructed above carries a canonical universal $1$-complex $\mathsf{G} \to \mathbb{Y}_{\mathsf{G}}$ together with a morphism $\mathsf{G} \hookrightarrow \Sigma \times \mathbb{Y}_{\mathsf{G}}$. However, these are not necessarily compatible with respect to the inclusions $\mathbb{Y}_{\mathsf{H}} \hookrightarrow \mathbb{Y}_{\mathsf{G}}$. As before, it is necessary to subdivide in order to guarantee compatible gluing. This produces, non-uniquely, a family of tropical expansions $\Upsilon \to \T$ as required. Given $\f \in |\T|$ the fibre $\Upsilon_\f$ is a polyhedral subdivision of the canonical fibre
\begin{equation*} \Upsilon_\f \to \mathsf{G}_\f \end{equation*}
obtained by introducing additional bivalent vertices (tube bubbling). One final subdivision is required in order to guarantee flatness of $\p \colon \Upsilon \to \T$ and this completes the construction.

 For the rest of this paper, we fix a choice of universal tropical expansion $\Upsilon \to \T$.

\begin{remark} \label{rmk: T not complete} The morphism $\Upsilon = (\Sigma \times \T)^\dag \to \Sigma \times \T$ is a subdivision in the generalised sense (see Definition~\ref{def: subdivision}). As a consequence, the associated universal target expansion will usually be non-proper over the associated moduli space of target expansions.

The universal target may be compactified, non-uniquely and at the cost of a further subdivision of $\T$ \cite[Remark 3.5.2]{MR20}. This amounts to a choice of factorisation of the non-complete subdivision through a complete subdivision (see \S \ref{sec: subdivisions}). A compactified universal target is unnecessary for applications to enumerative geometry. It can be a helpful conceptual crutch, but equally it can lead to bogus complications. Our results on the rubber action apply uniformly to non-compact and compact universal targets.\end{remark}

\subsubsection{Universal} The \textbf{moduli space of target expansions} 
\begin{equation*} \Exp=\Exp(X|D) \end{equation*}
is by definition the Artin fan associated to the moduli space of tropical expansions $\T$. To each cone $\tau \in \T$ we associate the Artin cone $\Acal_\tau = [U_\tau/T_\tau]$, where $U_\tau$ is the associated affine toric variety and $T_\tau$ is its dense torus. These are glued together according to the face inclusions of $\T$ in order to produce $\Exp$. Note that each Artin cone $\Acal_\tau$ is an open substack of $\Exp$.

The tropicalisation of $\Exp$ is just $\T$ again, and so the universal tropical expansion $\Upsilon$ defines a \textbf{universal target expansion}:
\begin{equation*} \Xfrak = (X \times \Exp)^\dagger \overset{p}{\longrightarrow} \Exp.\end{equation*}
The morphism $p \colon \Xfrak \to \Exp$ is representable, and flat with reduced fibres.

\subsubsection{Versal} Fix a cone $\tau \in \T$. The universal tropical expansion restricts to a tropical expansion
\begin{equation*} \Upsilon_\tau = (\Sigma \times \tau)^\dagger \overset{\p}{\longrightarrow} \tau \end{equation*}
over a single cone. Let $\Acal_\tau \subseteq \Exp$ be the corresponding open substack, and let $\Xfrak_{\tau}=\Xfrak|_{\Acal_\tau}$ be the restriction of the universal target. By definition this is the logarithmic modification of $X \times \Acal_\tau$ induced by $\Upsilon_\tau$:
\begin{equation*} \Xfrak_{\tau} = (X \times \Acal_\tau)^\dag \to  \Acal_\tau. \end{equation*}
There is an associated logarithmic modification on the prequotient:
\begin{equation*} \Xcal_{\tau} = (X\times U_\tau)^\dag \to U_\tau.\end{equation*}
This is a logarithmic expansion of $X$ over $U_\tau$ (\S \ref{sec: tropical expansions general case}). We refer to $\Xcal_{\tau}$ as a \textbf{versal expanded target} and $U_\tau$ as a \textbf{versal space of target expansions}.\medskip

\noindent \textbf{Observation.} The universal target is the quotient of the versal target by the rubber torus action described in \S \ref{sec: rubber action general case}:
\begin{equation} \Xfrak_\tau = [ \Xcal_{\tau} / T_\tau] \to [U_\tau / T_\tau] = \Acal_\tau .\end{equation}
This is the key observation which leads to the appearance of the rubber torus in the moduli problem.

\begin{remark}[Embedded $1$-complexes versus tropical stable maps] \label{rmk: T is not tropical stable maps} The tropical moduli space of target expansions $\T$ parametrises embedded $1$-complexes. In the setting of logarithmic Gromov--Witten theory, one might hope to replace $\T$ with a moduli space of tropical stable maps. However, this leads to a badly-behaved moduli problem: length parameters of contracted edges give rise to one-parameter families of rubber automorphisms which act trivially on the expanded target, forcing all such maps to be unstable and precluding properness of the moduli space. For expanded stable maps we must therefore also use embedded $1$-complexes to parametrise target expansions.\end{remark}

\subsection{Definition of expanded stable maps} With hindsight, the universal expansion $\Xfrak \to \Exp$ can be used to give an alternative description of the space of expanded stable maps, as an algebraic stack consisting of predeformable maps to the expanded target.

Here we define the objects and isomorphisms of this stack. It is not our goal to provide a comprehensive foundational treatment, and we defer the verification of key properties to future work. Our aims are more targeted, and focus on those aspects of the theory which are truly new: formulating the rubber action in terms of tropical geometry, and carefully explaining its relation to the $2$-morphisms in the universal expansion.

\begin{definition}\label{def: expanded map}
An \textbf{expanded map} to $(X|D)$ (over a base scheme $S$) is a $2$-commutative diagram of algebraic stacks
\bcd
C \ar[d,"\pi" left] \ar[r,"F"] & \Xfrak \ar[d,"p"] \\
S \ar[r,"G"] & \Exp
\ecd
such that $C \to S$ is a family of marked prestable curves and $F$ is predeformable, i.e. $F$ does not intersect codimension-$2$ strata of the expanded target, and intersects codimension-$1$ strata only at marked points and nodes, with equal tangencies on the two sides of every node.

 The marking sections of $\pi$ are suppressed in the notation. We denote by $\alpha$ the chosen $2$-isomorphism $\alpha \colon G \circ \pi \cong p \circ F$ (this is part of the data).
\end{definition}


\begin{definition} \label{def: isomorphisms of expanded maps 1} An \textbf{isomorphism} between expanded maps
\begin{equation*} (C_1 \to S_1, G_1,F_1, \alpha_1) \cong (C_2 \to S_2, G_2,F_2,\alpha_2) \end{equation*}
consists of the following data:
\begin{enumerate}
\item An isomorphism of marked source curves, i.e. a cartesian square in which both $\varphi$ and $\psi$ are isomorphisms:
\bcd
C_1 \ar[d,"\pi_1" left] \ar[r,"\varphi"] & C_2 \ar[d,"\pi_2"] \\
S_1 \ar[r,"\psi"] & S_2.
\ecd
\item $2$-isomorphisms $\lambda: G_1 \cong G_2 \circ \psi$ and $\mu: F_1 \cong F_2 \circ \varphi$.
\end{enumerate}
The $2$-isomorphisms must be compatible in the sense that the following diagram in the groupoid $\Hom(C_1,\Exp)$ commutes:
\begin{equation}\label{diag: commuting 2-morphism diagram}
\begin{tikzcd}
G_1 \circ \pi_1 \ar[d,"\pi_1^\star \lambda" left] \ar[rr,"\alpha_1"] & & p \circ F_1  \ar[d,"p_\star \mu"] \\
G_2 \circ \psi \circ \pi_1 \ar[equal]{r}   & G_2 \circ \pi_2 \circ \varphi \ar[r,"\varphi^\star \alpha_2"]& p \circ F_2 \circ \varphi.
\end{tikzcd}
\end{equation}
Since $p$ is representable, the choice of $\lambda$ determines the choice of $\mu$.
\end{definition}

In particular, an automorphism of an expanded map  $(C \to S, G,F, \alpha)$ is the data of an automorphism $\varphi\colon C\to C$ over $S$ and $2$-isomorphisms $\lambda \colon G \cong G$ and $\mu \colon F \cong F \circ \varphi$ such that the diagram \eqref{diag: commuting 2-morphism diagram} commutes. On the other hand, Definition~\ref{def: isomorphisms of expanded maps 1} clearly generalises to allow an arbitrary morphism $\psi \colon S_1 \to S_2$ on the base, defining the morphisms $(C_1 \to S_1, G_1, F_1, \alpha_1) \to (C_2 \to S_2, G_2, F_2, \alpha_2)$ in the stack of expanded maps.

There is a suitable notion of \textbf{stability} for expanded maps, similar to \cite[Definition 4.2.1]{MR20}. We discuss this in detail in \S \ref{sec: stability}, and show that it implies finiteness of automorphisms (but is not equivalent to it).

The above definitions make no use of logarithmic structures. Together, they produce a stack over the category of schemes.~To prove that this stack is algebraic, finite-type and proper, logarithmic structures are unnecessary. A model for the required arguments is given in \cite{MR20}. This mirrors and generalises J. Li's approach for smooth divisors, see Remark~\ref{rem: Jun Li comparison}.

Although the above definitions make no use of logarithmic structures, the resulting moduli space is still expected to carry a natural logarithmic structure. This will be obtained by pulling back from the moduli space of expanded maps to the Artin fan, the latter being irreducible with toroidal boundary. Toroidal and tropical techniques are still expected to play a central role.

\begin{remark}The space of expanded stable maps is expected to coincide, up to a quasi-finite morphism arising from saturation issues, with a space of logarithmic expanded stable maps. A logarithmic formulation will be useful for analysing the deformation theory. For smooth pairs, the role of the rubber action in the logarithmic setting has been explained in \cite{MarcusWise}.

On the other hand, it may be possible to build an obstruction theory directly, without appealing to logarithmic structures, along the lines of \cite{Li2} and the reformulation \cite[\S 3.2]{AbramovichMarcusWiseComparison}. This is because expanded stable maps are disjoint from the codimension-$2$ strata of the target.	
\end{remark}

\begin{remark} In logarithmic Donaldson--Thomas theory, predeformability is an open condition in the relative Hilbert scheme of $\Xfrak \to \Exp$. In that context the deformation theory may be easily described without reference to logarithmic structures. On the Gromov--Witten side, however, predeformability is not an open condition in the space of maps to $\Xfrak$. \end{remark}

\subsection{Pullback of the universal target}\label{sec: pullback of universal target} The data defining an expanded map is rather abstract, involving morphisms $F$ and $G$ to Artin stacks. We provide a more concrete description, by separating the data $(C \to S, G)$ from the data $(F, \alpha)$. We identify the pair $(F,\alpha)$ with the data of a morphism to a pullback of the universal target. This has the benefit that the target is a scheme.
\begin{lemma} \label{lem: stable map given by map to pullback} The data of an expanded map $(C \to S, G, F, \alpha)$ is equivalent to the data $(C \to S, G)$ together with a (predeformable) morphism $f$ fitting into a commuting triangle:
\bcd
C \ar[r,"f"] \ar[rd] & G^\star \Xfrak \ar[d] \\
& S.
\ecd
\end{lemma}
\begin{proof} This is tautological. By the construction of fibre products of stacks, an object $f \in \Hom(C,G^\star \Xfrak)$ is given by the data
\begin{equation}\label{eqn: data of map to fibre product} (\pi: C \to S, F \colon C \to \Xfrak, \alpha : G \circ \pi \cong p \circ F)\end{equation}
which is precisely the extra data $(F,\alpha)$ required to obtain an expanded map.
\end{proof}

Since $p \colon \Xfrak \to \Exp$ is representable, the pullback $G^\star \Xfrak$ is a scheme. We now explain concretely how the data $(F,\alpha)$ determines the map $f$ to this scheme. Consider an expanded map over a closed point:
\bcd
C \ar[d,"\pi" left] \ar[r,"F"] & \Xfrak \ar[d,"p"] \\
\Speck \ar[r,"G"] & \Exp.
\ecd
We use the notation established in \S \ref{sec: constructing universal target}. There is a unique cone $\tau \in \T$ such that the morphism $G \colon \Speck \to \Exp$ factors through the closed substack:
\begin{equation*} \Bcal T_{\tau} \subseteq \Acal_\tau \subseteq \Exp.\end{equation*}
The (restriction of the) universal target $\Xfrak_{\tau}$ is the quotient of the versal target:
\begin{equation*} \Xfrak_\tau = [ \Xcal_{\tau} / T_\tau] \to [U_\tau / T_\tau] = \Acal_\tau .\end{equation*}
Restricting over the closed substack $\Bcal T_{\tau} \subseteq \Acal_\tau$ gives the quotient
\begin{equation*} [Y_\tau/T_\tau] \end{equation*}
where $Y_\tau$ is the central fibre of the logarithmic expansion $\Xcal_{\tau}$. The original expanded map can be rewritten as
\begin{equation} \label{eqn: expanded map over point}
\begin{tikzcd}
C \ar[d,"\pi" left] \ar[r,"F"] & {[Y_\tau/T_\tau]} \ar[d,"p"] \\
\Speck \ar[r,"G"] & \Bcal T_\tau
\end{tikzcd}
\end{equation}
and it follows immediately that there is an isomorphism:
\begin{equation*}\label{eqn: isomorphism pullback universal target with expansion} \xi \colon G^\star \Xfrak \cong Y_\tau.\end{equation*}
However, there is not a unique choice for $\xi$. The collection of such isomorphisms is a $T_\tau$-torsor, the action given by composing with the rubber automorphisms of $Y_\tau$ (see Lemma~\ref{lem: aut is rubber} below).

Let $R$ be the $T_\tau$-torsor on $\Speck$ corresponding to the morphism $G$. Then a choice of isomorphism $\xi$ is equivalent to a choice of trivialisation of $R$. Choose one such trivialisation. Then the composite
\begin{equation*} f \colon C \to G^\star \Xfrak \xrightarrow{\xi} Y_\tau \end{equation*}
is well-defined. This is obtained from the data $(F,\alpha)$ as follows. The morphism $F$ is given by the data of a $T_{\tau}$-torsor $P$ over $C$ together with a $T_\tau$-equivariant morphism $P \to Y_\tau$. Since we have chosen a trivialisation of $R$, the isomorphism $\alpha\colon \pi^\star R \cong P$ induces a trivialisation of $P$, i.e. a section $C \to P$. Composing produces the desired morphism $f \colon C \to Y_\tau$.

\begin{remark}\label{rmk: map to quotient stack lifts} It is not the case that the groupoid of maps to $[Y_\tau/T_\tau]$ is equivalent to the groupoid of maps to $Y_\tau$ modulo the action of $T_\tau$. For the equivalence obtained above, it is essential to work relative to $\Bcal T_\tau$ and insist on the $2$-commutativity of \eqref{eqn: expanded map over point}. Without this, the map $F \colon C \to [Y_\tau/T_\tau]$ need not lift to any map $f \colon C \to Y_\tau$, as the $T_\tau$-torsor $P$ may not be trivial.\end{remark}


From now on, we will use Lemma~\ref{lem: stable map given by map to pullback} to describe an expanded map as the data $(C\to S, G, f)$. We now wish to describe the isomorphisms between such objects. We require the following simple observation.
\begin{lemma} \label{lem: 2-morphism induces isomorphism of pullback} Let $p \colon \Xfrak \to \Zcal$ be a morphism of stacks and consider two morphisms:
\bcd
S \arrow[r,shift left,"H_1"] \arrow[r,shift right,"H_2" below] & \Zcal.
\ecd
Then every $2$-isomorphism $\lambda : H_1 \cong H_2$ induces an isomorphism $\lambda_\dag \colon H_1^\star \Xfrak \to H_2^\star \Xfrak$.
\end{lemma}
\begin{proof}This is abstract nonsense. The objects of $H_i^\star \Xfrak$ over a test scheme $R$ are given by
\begin{equation*} (\pi \colon R \to S, F \colon R \to \Xfrak, \alpha \colon H_i \circ \pi \cong p \circ F)\end{equation*}
and so $\lambda_\dag \colon H_1^\star \Xfrak \to H_2^\star \Xfrak$ is defined by:
\begin{equation*} (\pi, F, \alpha) \mapsto (\pi, F, \alpha \circ \pi^\star \lambda^{-1}).\qedhere \end{equation*}\end{proof}

\begin{lemma} \label{lem: two notions of isomorphism are the same} An isomorphism of expanded maps
\begin{equation*} (C_1 \to S_1, G_1, f_1) \cong (C_2 \to S_2, G_2, f_2)\end{equation*}
as defined in Definition~\ref{def: isomorphisms of expanded maps 1}, is equivalent to the following data:
\begin{enumerate}
\item An isomorphism of source curves, as in Definition \ref{def: isomorphisms of expanded maps 1} (1).
\item A $2$-isomorphism $\lambda \colon G_1 \cong G_2 \circ \psi$.
\end{enumerate}
This data is required to be compatible with the maps $f_i$ in the sense that the following diagram (of schemes) commutes
\begin{equation} \label{diag: commuting rubber diagram}
\begin{tikzcd}
C_1 \ar[d,"\varphi" left] \ar[r,"f_1"] & G_1^\star \Xfrak \ar[d,"\Psi \circ \lambda_\dag"] \\
C_2 \ar[r,"f_2"] & G_2^\star \Xfrak
\end{tikzcd}
\end{equation}
where $\Psi$ is the canonical isomorphism $\psi^\star G_2^\star \Xfrak \to G_2^\star\Xfrak$. 
\end{lemma}
\begin{proof} This is an unraveling of definitions. The commutativity of \eqref{diag: commuting rubber diagram} is equivalent to the commutativity of \eqref{diag: commuting 2-morphism diagram}. \end{proof}

\begin{remark}\label{rem: Jun Li comparison} In the case of a smooth divisor, Lemmas~\ref{lem: stable map given by map to pullback} and \ref{lem: two notions of isomorphism are the same} show that the stack of expanded maps is equivalent to that constructed in \cite[\S 3]{Li1}. See also Example~\ref{example: rank 1}. \end{remark}

\subsection{Static target and the rubber action} \label{sec: rubber action versal target} \label{sec: static target} For families of maps whose expanded target does not vary, we can use Lemma~\ref{lem: two notions of isomorphism are the same} to reformulate isomorphisms in terms of the rubber action. This will be used in \S \ref{sec: boundary of moduli} to describe boundary strata of the moduli space.

Consider an expanded map $(C \to S, G, f)$. Suppose that there is a cone $\tau \leq \T$ of the tropical moduli space of target expansions such that $G$ factors through the closed substack:
\begin{equation*} \Bcal T_{\tau} \subseteq \Acal_\tau \subseteq \Exp.\end{equation*}
This always holds if $S=\Speck$. More generally it holds if the family of expanded maps is contained in a single locally-closed stratum of the moduli space. Following the discussion in \S \ref{sec: pullback of universal target}, restricting the universal target to this closed substack gives:
\bcd
{[ Y_\tau/T_\tau ]} \ar[d] \ar[r] \ar[rd,phantom,"\square"] & \Xfrak \ar[d] \\
\Bcal T_\tau \ar[r] & \Exp.
\ecd
We see that the pullback $G^\star \Xfrak$ is a locally trivial fibration over $S$ with fibre $Y_\tau$. The global twisting of this fibration is determined by the $T_\tau$-torsor $P$ giving the map $G \colon S \to \Bcal T_\tau$. 
\begin{lemma} \label{lem: pullback of universal target is mixing bundle} There is a natural isomorphism between $G^\star \Xfrak$ and the fibre bundle obtained from $P$ via the mixing construction:
\begin{equation*} G^\star\Xfrak = (P \times Y_\tau)/T_\tau.\end{equation*} \end{lemma}
\begin{proof} We compare functors of points. The desired identification is one of $S$-stacks, so let $\rho \colon R \to S$ be an arbitrary test scheme. Morphisms $R \to G^\star \Xfrak$ over $S$ are given by the following data: a $T_\tau$-torsor $Q$ over $R$, a $T_\tau$-equivariant morphism $Q \to Y_\tau$ and an isomorphism $Q \cong \rho^\star P$. 

For any scheme $U$ let $Y_{\tau}|_{U} = Y_\tau \times U$ viewed as a $U$-scheme. The data of the morphism $Q \to Y_\tau$ above is equivalent to the data of a morphism of $R$-schemes:
\begin{equation*} Q \to Y_\tau|_R = \rho^\star (Y_\tau|_S).
	\end{equation*}
Together with the isomorphism $Q \cong \rho^\star P$, this is equivalent to the data of a $T_\tau$-equivariant morphism of $R$-schemes
\begin{equation*} Q \to \rho^\star P \times_R Y_\tau|_R = \rho^\star (P \times_S Y_\tau|_S) = \rho^\star (P \times Y_\tau)\end{equation*}
which is precisely the data of an $S$-morphism $R \to (P \times Y_\tau)/T_\tau$ as claimed.
\end{proof}

\begin{remark} If $P$ is isomorphic to the trivial torsor (for instance if $S=\Speck$, as in \S \ref{sec: pullback of universal target}) then we have $(P \times Y_\tau)/T_\tau \cong Y_\tau|_S$. However, this isomorphism is not unique: it depends on a choice of trivialisation of $P$.\end{remark}
From Lemma~\ref{lem: pullback of universal target is mixing bundle} we immediately conclude:
\begin{lemma} \label{lem: expanded map to static target as map to mixing bundle}The data of the expanded map $(C \to S, G, f)$ is equivalent to the data of a $T_\tau$-torsor $P \to S$ together with a commutative diagram:
\bcd
C \ar[r,"f"] \ar[d,"\pi" left] & (P \times Y_\tau)/T_\tau \ar[ld] \\
S.
\ecd
\end{lemma}
We now discuss isomorphisms. Consider two expanded maps over the same base $S$ and with the same morphism $G$ to the stack of target expansions:
\begin{equation*} (C_1 \to S,G, f_1), (C_2 \to S,G, f_2). \end{equation*}
We assume as above that $G$ factors through a closed substack $\Bcal T_\tau \subseteq \Exp$ and let $P$ be the associated $T_\tau$-torsor. Writing $\Ycal_\tau=(P \times Y_\tau)/T_\tau$, we use the previous lemma to present these maps as:
\begin{equation*} (C_1 \to S, P, f_1 \colon C_1 \to \Ycal_\tau), (C_2 \to S, P, f_2 \colon C_2 \to \Ycal_\tau).\end{equation*}
Note that there is a fibrewise rubber action $T_\tau \acts \Ycal_\tau$, defined by letting $T_\tau$ act trivially on $P$ and via the rubber action on $T_\tau$.

On the other hand, since $G$ factors through $\Bcal T_\tau$ the $2$-morphisms $\lambda \colon G \cong G$ are precisely the elements of the isotropy group $T_\tau$. We saw in Lemma~\ref{lem: 2-morphism induces isomorphism of pullback} that each $\lambda \in T_\tau$ induces an automorphism $\lambda_\dag \colon G^\star \Xfrak \cong G^\star \Xfrak$.

\begin{lemma}\label{lem: aut is rubber} Under the identification $G^\star \Xfrak = \Ycal_\tau$, the automorphism $\lambda_\dag$ on $G^\star \Xfrak$ coincides with the rubber action on $\Ycal_\tau$.\end{lemma}

\begin{proof} Fix a morphism $f \colon R \to G^\star \Xfrak$. We will use the notation established in the proof of Lemma~\ref{lem: pullback of universal target is mixing bundle} to describe the data giving $f$ and the equivalent morphism $f^\prime \colon R \to \Ycal_\tau$.

Take $\lambda \in T_\tau$. From the proof of Lemma~\ref{lem: 2-morphism induces isomorphism of pullback}, we see that that $\lambda_\dag\circ f$ is obtained by composing the isomorphism $Q \cong \rho^\star P$ with $\lambda^{-1} \colon \rho^\star P \cong \rho^\star P$. It follows that $\lambda_\dag \circ f^\prime$ is obtained by composing the $T_\tau$-equivariant morphism  $h \colon Q \to \rho^\star (P \times Y_\tau)$ with the automorphism:
\[ (\lambda^{-1},1) \colon \rho^\star (P \times Y_\tau) \to \rho^\star (P \times Y_\tau).\]
But the resulting morphism $R \to (P \times Y_\tau)/T_\tau = \Ycal_\tau$ is equal (i.e. is uniquely isomorphic) to the morphism obtained by composing $h$ with the alternative automorphism:
\begin{equation} \label{eqn: rubber action on mixing bundle} (1,\lambda) \colon \rho^\star (P \times Y_\tau) \to \rho^\star (P \times Y_\tau). \end{equation}
This is because the action $T_\tau \acts P \times Y_\tau$ which we quotient by is given by $(\lambda,\lambda)$, and $(\lambda,\lambda)\circ(\lambda^{-1},1)=(1,\lambda)$. But \eqref{eqn: rubber action on mixing bundle} is precisely the automorphism inducing the rubber action on $\Ycal_\tau$.\end{proof}

\begin{theorem} \label{thm: isomorphism of expanded maps in restricted case} An isomorphism of expanded maps
\begin{equation*} (C_1 \to S, P, f_1 \colon C_1 \to \Ycal_\tau) \cong (C_2 \to S, P, f_2 \colon C_2 \to \Ycal_\tau) \end{equation*}
is equivalent to the following data:
\begin{enumerate}
\item An isomorphism $\varphi \colon C_1 \to C_2$ over $S$.
\item An element $\lambda \in T_\tau$.
\end{enumerate}
This data is required to satisfy $\lambda \circ f_1 = f_2 \circ \varphi$ where $\lambda \colon \Ycal_\tau \cong \Ycal_\tau$ is the rubber action.\end{theorem}
\begin{proof} This follows by combining Lemma~\ref{lem: two notions of isomorphism are the same} and Lemma~\ref{lem: aut is rubber}.\end{proof}

\begin{remark} More generally, given expanded maps over different bases, we define in a similar way isomorphisms covering a given base isomorphism $\psi \colon S_1 \to S_2$. In this case, $\lambda$ is a choice of isomorphism $P_1 \cong \psi^\star P_2$, and the set of these forms a torsor over the rubber torus $T_\tau$.\end{remark}

\begin{remark} \label{rmk: can also do DT} Similar results apply to the logarithmic Donaldson--Thomas moduli spaces constructed in \cite[\S 4]{MR20}. The rubber torus is not discussed explicitly there, though it can be found lurking whenever the categorical equivalence between Artin fans and cone complexes is invoked. Our description of the action in terms of tropical position maps is new.\end{remark}

\subsection{Stability} \label{sec: stability}
Stability for expanded maps is defined analogously to stability for expanded ideal sheaves \cite[\S\S4.1--4.2]{MR20}. Consider an expanded map over a closed point. By Lemma~\ref{lem: expanded map to static target as map to mixing bundle} we may (after choosing a trivialisation of the torsor $P$ over $\Speck$) represent this as
\[ (C \to \Speck, f \colon C \to Y_\tau) \]
\noindent for a unique cone $\tau \leq \T$. Recall from \S \ref{sec: tropical target} that over an interior point $\f \in |\tau|$ there is a polyhedral subdivision
\[ \Upsilon_\f \to \mathsf{G}_\f \]
of the fibre $\mathsf{G}_\f$ of the universal embedded $1$-complex, given by introducing additional bivalent vertices. These are referred to as \textbf{tube vertices}, and the corresponding components of $Y_\tau$ are referred to as \textbf{tube components}. Here it is important that we use the non-complete target $\Upsilon$.

\begin{definition} Consider an irreducible component $Y_v \subseteq Y_\tau$ corresponding to a bivalent vertex $v\leq\Upsilon_{\f}$ (which may or may not be a tube vertex). An expanded map to $Y_\tau$ is a \textbf{tube along $Y_v$} if and only if 
\[ f^{-1}(Y_v) \subseteq C \]
is a union of strictly semistable components mapping non-constantly into the fibres of $Y_v \to X_v$ with all components defining the same slope direction.

The slope direction associated to a semistable component of the source curve is defined as follows. The corresponding vertex of the tropical source curve is bivalent, and by the balancing condition in the toric fibres the slopes along the two adjacent edges must be equal. The slope direction is the linear subspace spanned by either of these slopes. It is a one-dimensional linear subspace of $N_{\sigma_v}\otimes \RR$.\end{definition}

\begin{definition}\label{def:stability} An expanded map $(C \to \Speck, f\colon C \to Y_\tau)$ is \textbf{stable} if it meets every stratum of $Y_\tau$, is stable in the usual sense, and is a tube exactly along the tube components of the target.\end{definition}

\begin{remark} Given a non-tube component $Y_v \subseteq Y_\tau$ the last condition states that either:
 \begin{enumerate}
 \item There exists an irreducible component $C_0 \subseteq C$ mapping to $Y_v$ such that the composite $C_0 \to Y_v \to X_v$ is stable in the usual sense.
 \item If every component mapping to $Y_v$ is strictly semistable and contained in a fibre of $Y_v \to X_v$, then there must exist two components $C_1,C_2 \subseteq C$ with linearly independent slopes $m_1, m_2 \in N_{\sigma_v}$.
\end{enumerate}
\end{remark}
We will show that this notion of stability implies, but is not equivalent to, finiteness of automorphisms. We begin with a preparatory lemma.
\begin{lemma} \label{lem: rubber acts nontrivially} The product of all tropical position maps
 \[ \prod_v \varphi_v \colon T_{\tau} \to \prod_v T_{\sigma_v} \]
 is injective, so that each element of $T_\tau$ acts nontrivially on at least one irreducible component of the expanded target.\end{lemma}
\begin{proof} This is an immediate consequence of the construction of the universal tropical target $\Upsilon \to \T$ using embedded $1$-complexes.\end{proof}

\begin{theorem} A stable expanded map has finite automorphism group. \end{theorem}
\begin{proof} Automorphisms are characterised by Theorem~\ref{thm: isomorphism of expanded maps in restricted case}. Assume for a contradiction that the automorphism group is infinite. Since the expanded map is stable in the usual sense, there must exist a one-parameter subgroup of the rubber torus
\[ T_\rho \subseteq T_\tau \]
such that for every $\lambda \in T_\rho$ there is an automorphism $\varphi_\lambda$ of $C$ with $\lambda \circ f = f \circ \varphi_\lambda$. Here $\rho$ is a one-dimensional conical subset:
\[ \rho \subseteq \tau .\]
By Lemma~\ref{lem: rubber acts nontrivially}, $T_\rho$ acts nontrivially on at least one target component. We claim that in fact it acts nontrivially on at least one non-tube component. Indeed, tube components arise from a subdivision $\Upsilon \to \mathsf{G}$ which does not alter the base cone $\tau$. Pulling back along $\rho \subseteq \tau$ produces an embedded $1$-complex parametrised by $\rho$, which by the construction of $\T$ contain a vertex $v$ not mapping to $0 \leq \Sigma$. It follows that $T_\rho$ acts non-trivially on $Y_v$. In fact, the action is free on the fibres of $Y_v^\circ \to X_v^\circ$.
 
From $\lambda \circ f = f \circ \varphi_\lambda$ it follows that every irreducible component of $C$ mapping to $Y_v$ must map its interior (i.e. the complement of all special points) into a fibre of the principal torus bundle $Y_v^\circ \to X_v^\circ$. Moreover every such component must have infinitely many automorphisms $\varphi_\lambda$, and so must be strictly semistable.
 
 Finally, if there are two such semistable components with linearly independent slopes, then we cannot have $\lambda \circ f = f \circ \varphi_\lambda$ since there is only a single one-dimensional subtorus bundle of $Y_v^\circ \to X_v^\circ$ which is preserved by $T_\rho$.
 
 We conclude that the expanded map is a tube along the non-tube component $Y_v$ of the target, contradicting stability.\end{proof}

\begin{remark} The previous argument mirrors the proof of \cite[Theorem~4.5.1]{MR20}. In particular, the argument does not require the explicit description of the rubber torus or its action in terms of tropical position maps. \end{remark}

 \begin{example}\label{ex:tubevertices} 
We conclude with an example to show that finiteness of automorphisms does not imply stability. Consider the following embedded $1$-complex:
 
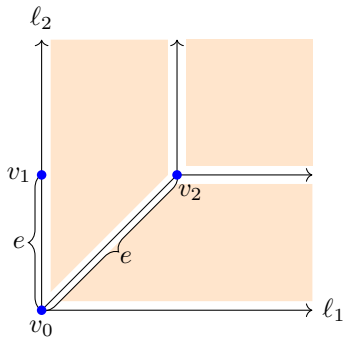
\begin{center}
	\begin{minipage}[t]{0.7\textwidth}
		\centering
		\begin{tikzpicture} [scale=.60]
		\tikzstyle{every node}=[font=\normalsize]
		\tikzset{arrow/.style={latex-latex}}
		
		\tikzset{cross/.style={cross out, draw, thick,
         minimum size=2*(#1-\pgflinewidth), 
         inner sep=1.2pt, outer sep=1.2pt}}

\coordinate (OP) at (0,0,0);
\coordinate (BP) at (6,0,0);
   \coordinate (CP) at (0,6,0);
 \coordinate (v1P) at (0,3,0);
      \coordinate (v2P) at (3,3,0);
            \coordinate (v2Plab) at (3.3,3,0);
      \coordinate (vl1) at (3,6,0);
      \coordinate (vl2) at (6,3,0);
    \coordinate (vl3) at (6,6,0);


\fill[orange!20] (0.2,0.4) -- (2.8,3) -- (2.8,6) -- (0.2,6) -- (0.2,0.4);
\fill [orange!20] (0.4,0.2) -- (6,0.2) -- (6,2.8) -- (3,2.8) -- (0.4,0.2);
\fill [orange!20] (3.2,3.2) -- (6,3.2) -- (6,6) -- (3.2,6) -- (3.2,3.2);

\draw [decorate,decoration={brace,amplitude=5pt},xshift=0.4pt,yshift=-0.4pt] ([c]OP)--([c]v1P) node[black,midway,xshift=-0.3cm,yshift=0 cm] {\small{$e$}};
\draw [decorate,decoration={brace,amplitude=5pt,mirror},xshift=0.4pt,yshift=-0.4pt] ([c]OP)--([c]v2P) node[black,midway,xshift=0.22cm,yshift=-0.2 cm] {\small{$e$}};

\begin{scope}[every coordinate/.style={shift={(0,0,0)}}]
\draw[->] ([c]OP)--([c]BP);
\draw [->]([c]OP)--([c]CP);

\draw[->] ([c]v2P)--([c]vl1);
\draw[->] ([c]v2P)--([c]vl2);
\draw ([c]v2P)--([c]OP);
   \fill[blue]([c]OP) circle (3pt);
       \fill[blue] ([c]v1P) circle (3pt);
       \fill[blue] ([c]v2P) circle (3pt);
       
       \node at ([c]OP) [below] {\small{$v_0$}};	
     \node at ([c]v1P) [left] {\small{$v_1$}};
     \node at ([c]v2Plab) [below] {\small{$v_2$}};
     
      \node at ([c]BP) [right] {\small{$\ell_1$}};
          \node at ([c]CP) [above] {\small{$\ell_2$}};

\end{scope}


\end{tikzpicture}
	\end{minipage}
\end{center}
Here the orange shaded regions are not included in the conical or polyhedral subdivisions. Note that there is a single tropical parameter $e$, even though the cone $\mathbb{X}_{\mathsf{G}}$ associated to this embedded $1$-complex has two tropical parameters. This expanded target appears in the universal tropical expansion $\Upsilon \to \T$ whenever $\mathbb{X}_{\mathsf{G}}$ is subdivided along its diagonal during the construction of $\T$.

Now consider a map to this expanded target which intersects every stratum and is a tube along the non-tube component $Y_{v_1}$. This map has finitely many automorphisms because the rubber torus $\Gm \langle e \rangle$ acts on $Y_{v_1}$ and $Y_{v_2}$ simultaneously, and the stability of the source curve over $Y_{v_2}$ is sufficient to guarantee finiteness of automorphisms.

On the other hand, the map is not stable. We conclude that stability is strictly stronger than finiteness of automorphisms. In this example we see that, without stability, the moduli space is not separated: if the map above arises as the limit of a one-parameter family, we can produce a different limit by simply collapsing the component $Y_{v_1}$ in the target and its corresponding tube component in the source. The interaction between stability and separatedness is discussed in \cite[Theorem~4.6.1]{MR20}. For smooth pairs this subtlety does not arise: stability and finiteness of automorphisms are equivalent \cite[Lemma 3.2]{Li1}. \end{example}

\section{The boundary of moduli} \label{sec: boundary of moduli}

\noindent We now combine the results of the previous two sections to describe the locally-closed boundary strata in the space of expanded stable maps (Theorem~\ref{thm: description of locally closed stratum}). We then discuss the remaining obstacles to a recursive description of the \emph{closed} boundary strata. Finally, we show how our perspective can be used to define higher-rank rubber spaces.

\subsection{Locally-closed strata}\label{sec: boundary of moduli locally-closed strata} Denote by $\Kup_\Gamma(X|D)$ the moduli space of expanded stable maps to $(X|D)$ with fixed discrete data $\Gamma$. As discussed, this depends on a choice of universal tropical expansion, which we suppress from the notation. There is a natural morphism:
\begin{equation*} \pi \colon \Kup_{\Gamma}(X|D) \to \Exp. \end{equation*}
Each cone $\tau  \leq \T$ corresponds to a locally-closed stratum $\Bcal T_\tau \subseteq \Exp$. We consider its preimage in $\Kup_\Gamma(X|D)$:
\begin{equation*} \Kup_\tau^\circ \colonequals \pi^{-1}(\Bcal T_\tau) \subseteq \Kup_\Gamma(X|D).\end{equation*}
The locus $\Kup_\tau^\circ \subseteq \Kup_\Gamma(X|D)$ consists precisely of those expanded stable maps $(C \to S, G, f)$ for which $G$ factors through $\Bcal T_\tau \subseteq \Exp$. This forms a locally-closed substack, over which the expanded target is static. We will refer to it as a locally-closed boundary stratum, although strictly speaking it is a union of finitely many locally-closed boundary strata, indexed by combinatorial types of tropical stable maps with the same image.

We let $\mathscr{Y}_\tau \to \Bcal T_\tau$ denote the restriction of the universal target to $\Bcal T_\tau$. Using notation from \S \ref{sec: constructing universal target}, we note that this differs from $\mathfrak{X}_\tau \to \Acal_\tau$, which is the restriction of the universal target to the \emph{open} substack $\Acal_\tau \subseteq \Exp$. In fact $\mathscr{Y}_\tau$ is the central fibre of $\Xfrak_\tau$, and from the identification $\Xfrak_\tau=[\Xcal_{\tau}/T_\tau]$ we obtain
\begin{equation*} \mathscr{Y}_\tau = [Y_\tau / T_\tau] \to \Bcal T_\tau \end{equation*}
where $Y_\tau \subseteq \Xcal_{\tau}$ is the central fibre of the versal target and $T_\tau \acts Y_\tau$ is the rubber action described in Theorem~\ref{prop: rubber action general}. The following definitions give rise to a moduli space of expanded rubber maps to $Y_\tau$. 
\begin{definition}\label{def: rubber map}
An expanded rubber map to $Y_\tau$ (over a base scheme $S$) consists of a $T_\tau$-torsor $P$ over $S$ and a commutative diagram of schemes
\begin{equation*}
\begin{tikzcd}
C \ar[d,"\pi" left] \ar[r,"f"] & \Ycal_\tau \ar[ld]  \\
S 
\end{tikzcd}
\end{equation*}
where $\Ycal_\tau = (P \times Y_\tau)/T_\tau$ is the associated $Y_\tau$-bundle, $C \to S$ is a family of prestable curves and $f$ is predeformable. \textbf{Stability} is defined in the same way as it is for expanded maps to $(X|D)$, see \S \ref{sec: stability} for details.
\end{definition}
\begin{definition} \label{def: isomorphism of rubber maps} An isomorphism between expanded rubber maps
\begin{equation*} (C_1 \to S_1, f_1) \cong (C_2 \to S_2, f_2) \end{equation*}
consists of the following data:
\begin{enumerate}
\item An isomorphism of source curves, i.e. a commutative square in which both $\varphi$ and $\psi$ are isomorphisms:
\bcd
C_1 \ar[d,"\pi_1" left] \ar[r,"\varphi"] & C_2 \ar[d,"\pi_2"] \\
S_1 \ar[r,"\psi"] & S_2.
\ecd
\item An element $\lambda \in T_\tau$ of the rubber torus.
\end{enumerate}	
This data is required to satisfy $\lambda \circ f_1 = f_2 \circ \varphi$, where the action $T_\tau \acts \Ycal_\tau$ is induced by the trivial action on $P$ and the rubber action on $Y_\tau$ (as in \S \ref{sec: static target}).
\end{definition}
\begin{theorem} \label{thm: description of locally closed stratum} The stratum $\Kup_{\tau}^\circ \subseteq \Kup_{\Gamma}(X|D)$ is isomorphic to the moduli space of expanded rubber maps to $Y_\tau$.\end{theorem}
\begin{proof} By definition $\Kup_\tau^\circ$ consists precisely of expanded stable maps for which the base morphism $G \colon S \to \Exp$ factors through $\Bcal T_\tau \subseteq \Exp$. The result then follows from Theorem~\ref{thm: isomorphism of expanded maps in restricted case}. \end{proof}
 
\begin{remark} Note that $Y_\tau$ is almost always reducible. In this case, the source curve of an expanded rubber map to $Y_\tau$ must also be reducible. As already mentioned, $\Kup_\tau^\circ$ is a union of finitely many locally-closed boundary strata, indexed by combinatorial types of tropical stable maps with the same image. Correspondingly, the space of expanded rubber maps to $Y_\tau$ with discrete data $\Gamma$ also decomposes into a finite union of strata. Theorem~\ref{thm: description of locally closed stratum} can be improved to describe these strata individually, by enhancing the discrete data $\Gamma$ to include the combinatorial type of the source curve. This is similar to the construction of spaces of stable maps indexed by stable marked graphs \cite{BehrendManin}.\end{remark}

\subsection{Recursive description of the boundary} \label{sec: recursive description of boundary} \label{sec: obstacles for recursive description} Theorem~\ref{thm: description of locally closed stratum} provides the first step towards a recursive description of the boundary of the moduli space of expanded stable maps. Two major tasks remain:

\begin{enumerate}
\item Describe the space of expanded rubber maps to $Y_\tau$ as a fibre product of spaces with smaller discrete data. Here a fundamental complication arises which does not occur for smooth pairs: the rubber torus can act nontrivially on a divisor joining two components of the expanded target.

It follows that the rubber torus cannot be split into a direct sum of tori such that each factor acts nontrivially on only a single irreducible component of $Y_\tau$. Consequently, it is not possible to express the stratum $\Kup_\tau^\circ$ as a fibre product of spaces of rubber maps to the individual irreducible components of $Y_\tau$.\medskip

\item There are certain simple target expansions for which the rubber torus happens to act trivially on the join divisors. In these cases, Theorem~\ref{thm: description of locally closed stratum} can be used to give a recursive description of the locally-closed stratum $\Kup_\tau^\circ$.

Even here, however, describing the \emph{closed} stratum $\Kup_\tau$ is considerably more difficult, because further target degenerations in the different factors of the fibre product may not be compatible. A similar problem arises in the expanded version of the degeneration formula, and its solution requires an effective description of the relative diagonal of the expanded join divisor \cite[\S\S 5-6]{RangExpansions}.
\end{enumerate}
We mention some speculative strategies for addressing these issues. For (1), a promising approach is to use rigidification \cite[\S 1.5.3]{MaulikPandharipande} to remove the rubber torus, thus circumventing the difficulties caused by its action on the join divisors. For (2), a potential technique is to pass to smaller birational models of the expanded moduli spaces (e.g. quasimap spaces), guaranteeing simpler geometry of the expanded join divisor.

In all cases, a necessary prerequisite will be a choice of ideal coordinate system on the rubber torus: see the discussion in Example~\ref{example: smooth divisor}.

\subsection{Expanded rubber spaces} \label{sec: rubber spaces} We now use the approach developed in \S \ref{sec: 2-morphisms} to define spaces of expanded maps to higher-rank rubber targets. The virtual fundamental classes of these spaces will produce toric contact cycles \cite{DhruvProduct,HPS,MolchoRanganathan, HolmesSchwarz}.

Consider a normal crossings pair $(X|D)$ together with a torus action $T \acts (X|D)$, i.e. an action $T \acts X$ which sends $D$ to itself. The basic example is a toric variety (or more generally, a toric variety bundle) together with the action of the dense torus, or a subtorus thereof.

Construct as in \S \ref{sec: constructing universal target} a universal tropical expansion $\Upsilon \to \T$ for $(X|D)$, with associated universal target expansion $\Xfrak \to \Exp$. Recall that $\Exp$ is covered by open substacks $\Acal_\tau$ for $\tau \leq \T$ over which the universal target is the quotient of the versal target by the rubber action:
\begin{equation*} \Xfrak_\tau = [\Xcal_{\tau}/T_\tau] \to [U_\tau/T_\tau] = \Acal_\tau.\end{equation*}
Recall that $\Xcal_{\tau}$ is a logarithmic modification of $X \times U_\tau$. The action $T \acts X \times U_\tau$ (given by the trivial action on the second factor) lifts to an action $T \acts \Xcal_{\tau}$ (covering the trivial action on $U_\tau$). We define the rubber target as the quotient of the versal target by $T \times T_\tau$:
\begin{equation*} \Rub(\Xfrak_\tau,T) \colonequals [\Xcal_{\tau}/(T \times T_\tau)] \to [U_\tau/(T \times T_\tau)] = \Bcal T \times \Acal_\tau \equalscolon \Rub(\Acal_\tau,T).\end{equation*}
These glue along the face inclusions of $\T$ to give a \textbf{rubber universal target}:
\begin{equation} \operatorname{Rub}(\Xfrak,T) \to \Bcal T \times \Exp = \Rub(\Exp,T).\end{equation}
Note that the projection morphism is representable. Using this, we define \textbf{expanded rubber maps} by repeating Definitions~\ref{def: expanded map} and \ref{def: isomorphisms of expanded maps 1}, replacing the universal target $\Xfrak \to \Exp$ by the rubber universal target $\Rub(\Xfrak,T) \to \Rub(\Exp,T)$.

The arguments in \S \ref{sec: pullback of universal target}, \S\ref{sec: static target} and \S\ref{sec: boundary of moduli locally-closed strata} then apply \emph{mutatis mutandis}, identifying rubber maps with maps to the pullback of the universal target, and describing the locally-closed strata of the resulting moduli space. In particular, on the locus where the target does not expand we have $\Rub(\Acal_\tau,T)=\Bcal T$ and a rubber map consists of a $2$-commutative diagram:
\bcd
C \ar[d] \ar[r] & {[X/T]} \ar[d] \\
S \ar[r] & \Bcal T.
\ecd
The $2$-commutativity ensures that $C \to [X/T]$ lifts fibrewise to a map $C \to X$ (see Remark~\ref{rmk: map to quotient stack lifts} and Lemma~\ref{lem: expanded map to static target as map to mixing bundle}), while the isotropy in $\Bcal T$ implies that maps differing by the rubber action are identified. As we move to the boundary of the moduli space, we must contend with both the global rubber torus $T$ and the local rubber torus $T_\tau$ associated to the given target expansion.

In situations where a recursive description of the boundary of expanded stable maps is possible (see \S \ref{sec: recursive description of boundary} above), the factors appearing in the fibre product will be spaces of expanded rubber maps.

\section{Examples} \label{sec: examples}
\noindent Throughout we will write $\RR_{\geq 0}^k \langle e_1, \ldots, e_k \rangle$ to indicate a cone coordinatised by $e_1,\ldots,e_k$.

\begin{example}\label{ex:rank1} \label{example: rank 1} \label{example: smooth divisor} \label{example: rank one} We begin with the rank one case. Let $(X|D)$ be a smooth pair with tropicalisation:
\[ \Sigma=\RR_{\geq 0}\langle \ell \rangle.\]
We consider the polyhedral subdivision of $\Sigma$ obtained by introducing two bounded edges with lengths $e_1,e_2$. This gives a tropical expansion with base cone:
\[ \tau=\RR^2_{\geq 0}\langle e_1 ,e_2 \rangle.\]
The height-$1$ slice of the conical subdivision $\Upsilon = (\Sigma \times \tau)^\dag$ is illustrated below, alongside the polyhedral subdivision $\p^{-1}(\f)$ for $\f \in |\tau|$ an interior point.

\begin{figure}[htb]
	\centering
		\centering
		\begin{tikzpicture} [scale=.70]
		\tikzstyle{every node}=[font=\normalsize]
		\tikzset{arrow/.style={latex-latex}}
		
		\tikzset{cross/.style={cross out, draw, thick,
         minimum size=2*(#1-\pgflinewidth), 
         inner sep=1.2pt, outer sep=1.2pt}}

   \coordinate (Otarg) at (0,-3,0);
         \coordinate (v1targ) at (1.3,-3,0);
          \coordinate (v3targ) at (3.5,-3,0);
            \coordinate (y3targ) at (6,-3,0);

\coordinate (O) at (0,0,0);
	 \coordinate (B) at (6,0,0);
          \coordinate (C) at (3,5,0);
          \coordinate (v1) at (1.5,2.5,0);
           \coordinate (v2) at (4.5,2.5,0);

	\coordinate (O) at (0,0,0);
	 \coordinate (B) at (6,0,0);
          \coordinate (C) at (3,5,0);
          \coordinate (v1) at (1.5,2.5,0);
           \coordinate (v2) at (4.5,2.5,0);
           
           \coordinate (lab1) at (3,0,0);
	 \coordinate (lab2) at (3,1.5,0);
	 	 \coordinate (lab3) at (3,2.5,0);

\draw[blue,thick] (v1)--(B);
\draw[blue,thick] (v1)--(v2);

\foreach \x in {O,B,C,v1,v2}
   \fill (\x) circle (2pt);
   
\draw[thick] (O)--(B);
\draw (B)--(C);
\draw (O)--(C);
\node at (O) [below] {\small{$e_1$}};	
     \node at (B) [below] {\small{$e_2$}};
     \node at (C) [above] {\small{$\ell$}};
     
     \node at (lab1) [below] {\small{$v_0$}};	
     \node at (lab2) [below] {\small{$v_1$}};
     \node at (lab3) [above] {\small{$v_2$}};

\begin{scope}[every coordinate/.style={shift={(8,5,0)}}]


\draw [decorate,decoration={brace,amplitude=5pt,mirror},xshift=0.4pt,yshift=-0.4pt] ([c]Otarg)--([c]v1targ) node[black,midway,xshift=0 cm,yshift=-0.4cm] {\small{$e_1$}};
\draw [decorate,decoration={brace,amplitude=5pt,mirror},xshift=0.4pt,yshift=-0.4pt] ([c]v1targ)--([c]v3targ) node[black,midway,xshift=0 cm,yshift=-0.4cm] {\small{$e_2$}};

\draw ([c]Otarg)--([c]v1targ);
\draw ([c]v1targ)--([c]v3targ);
\draw[->] ([c]v3targ)--([c]y3targ);
   \fill[blue] ([c]Otarg) circle (3pt);
       \fill[blue] ([c]v1targ) circle (3pt);
       \fill [blue]([c]v3targ) circle (3pt);
       \node at ([c]Otarg) [above] {\small{$v_0$}};	
     \node at ([c]v1targ) [above] {\small{$v_1$}};
     \node at ([c]v3targ) [above] {\small{$v_2$}};
         \node at ([c]y3targ) [right] {\small{$\ell$}};

\end{scope}

\end{tikzpicture}
\caption{Rank one expansion with two tropical parameters.}
\label{rank1}
\end{figure}
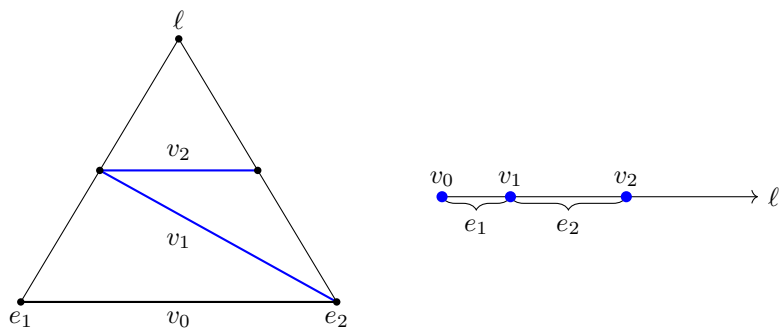

\noindent This produces a family of expansions $\Xcal_{\Upsilon} \to \Aaff^{\!2}$ whose central fibre $Y_\Upsilon \subseteq \Xcal_{\Upsilon}$ has three irreducible components, namely $Y_{v_0} = X$ and $Y_{v_1} = Y_{v_2} = \PP_D(\mathcal O_D(D) \oplus\mathcal O_D)$. The rubber torus is coordinatised by the tropical parameters on the base:
\begin{equation*} T_\tau = \tau \otimes \Gm = \Gm^2 \langle e_1 ,e_2 \rangle.\end{equation*}
The tropical position maps record the positions of the vertices $v_i$ in terms of the edge lengths $e_j$, and are given by:
\begin{align*}
\varphi_{v_0}(e_1,e_2) & = 0,\\
\varphi_{v_1}(e_1,e_2) & = e_1, \\
\varphi_{v_2}(e_1,e_2) & = e_1+e_2.	
\end{align*}
These induce homomorphisms of tori:
\begin{align*} \varphi_{v_0} \otimes \Gm & \colon \Gm^2 \langle e_1,e_2 \rangle \to 1, \quad \qquad \ (e_1,e_2) \mapsto 1, \\
\varphi_{v_1} \otimes \Gm & \colon \Gm^2 \langle e_1 , e_2 \rangle \to \Gm \langle \ell \rangle, \quad (e_1,e_2) \mapsto e_1, \\
\varphi_{v_2} \otimes \Gm & \colon \Gm^2 \langle e_1, e_2 \rangle \to \Gm \langle \ell \rangle, \quad (e_1,e_2) \mapsto e_1 e_2.	
\end{align*}
By Theorem~\ref{prop: rubber action general}, the rubber action is induced by these homomorphisms. The action is trivial on $Y_{v_0}$ and has weights $(1,0)$ and $(1,1)$ in the fibres of the $\PP^1$ bundles $Y_{v_1}$ and $Y_{v_2}$ respectively. Note in particular that the $1$-parameter subgroup $\Gm \langle e_1 \rangle \subseteq \Gm^2 \langle e_1 , e_2 \rangle$ acts nontrivially on both $Y_{v_1}$ and $Y_{v_2}$.

The action is trivial on the two join divisors. Following \S \ref{sec: locally closed strata toric case}, this may be seen by considering the generalised tropical position map $\varphi_P \colon \tau \to \theta_P$ where $P$ is one of the two bounded edges. We have $\theta_P=0$ and so the rubber action is trivial. An understanding of the global geometry of the irreducible components $Y_{v_i}$ is not required in order to describe the rubber action on the locally-closed strata.

The above description of the rubber action departs from the classical one, where the torus splits into one-dimensional factors which each act nontrivially on a single component of the target. We recover the classical description by changing coordinates on the rubber torus from $(e_1,e_2)$ to $(e_1,e_1e_2)$, thus ``diagonalising'' the action.

This change of coordinates is necessary in order to obtain a recursive description of the boundary of the moduli space of expanded stable maps. It is unclear whether there is a similar ``best choice'' of rubber torus coordinates for higher rank targets. On strata where the action is trivial on all join divisors, we expect the existence of suitable coordinates which split the torus and facilitate a recursive description of the boundary. However, as we will see in the next example, the rubber torus often acts nontrivially on the join divisors, precluding such a simple splitting formalism (see also the discussion in \S \ref{sec: recursive description of boundary}).

\end{example}

\begin{example}\label{ex:rank2asfan+actionondiv}
Consider a simple normal crossings pair $(X| D=D_1+D_2)$ with tropicalisation:
\[ \Sigma= \RR^2_{\geq 0} \langle \ell_1 , \ell_2 \rangle.\]
We consider the following tropical expansion of $\Sigma$ over the base cone $\tau = \RR_{\geq 0} \langle e \rangle$. As before, the conical subdivision is represented via its height-$1$ slice.

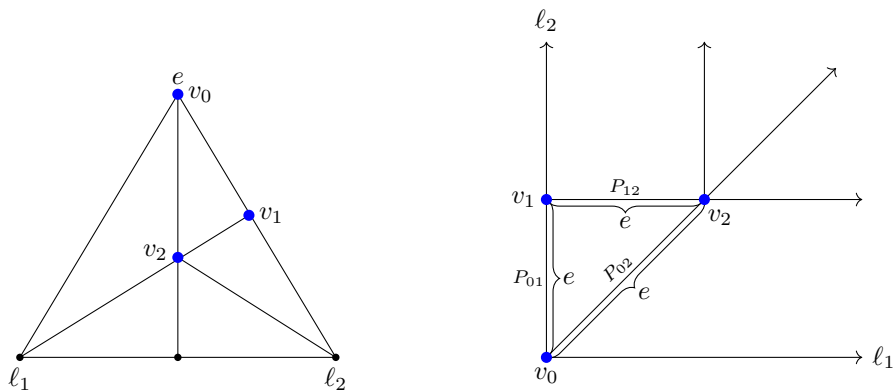
\begin{figure}[htb]
	\centering
		\centering
		\begin{tikzpicture} [scale=.7]
		\tikzstyle{every node}=[font=\normalsize]
		\tikzset{arrow/.style={latex-latex}}
		
		\tikzset{cross/.style={cross out, draw, thick,
         minimum size=2*(#1-\pgflinewidth), 
         inner sep=1.2pt, outer sep=1.2pt}}

\coordinate (O) at (0,0,0);
	 \coordinate (B) at (6,0,0);
	  \coordinate (Bhalf) at (3,0,0);
   \coordinate (C) at (3,5,0);
           \coordinate (v1) at (4.35,2.7,0);
          \coordinate (v2) at (3,1.9,0);
		\coordinate (v2Navid) at (3,2,0);

\coordinate (OP) at (0,0,0);
\coordinate (BP) at (6,0,0);
   \coordinate (CP) at (0,6,0);
 \coordinate (v1P) at (0,3,0);
      \coordinate (v2P) at (3,3,0);
            \coordinate (v2Plab) at (3.3,3,0);
      \coordinate (vl1) at (3,6,0);
      \coordinate (vl2) at (6,3,0);
    \coordinate (vl3) at (5.5,5.5,0);

\coordinate (P01) at (0.17,1.5,0);
\coordinate (P02) at (1.35,1.7,0);
\coordinate (P12) at (1.5,2.9,0);

   
\draw (O)--(B);
\draw (B)--(C);
\draw (Bhalf)--(C);
\draw (O)--(C);
\draw(v1)--(O);
\draw(v2)--(B);

\node at (O) [below] {\small{$\ell_1$}};	
     \node at (B) [below] {\small{$\ell_2$}};
     \node at (C) [above] {\small{$e$}};
     
     \node at (C) [right] {\small{$v_0$}};	
     \node at (v1) [right] {\small{$v_1$}};
	\node at (v2Navid) [left] {\small{$v_2$}};

\foreach \x in {O,B, Bhalf}
   \fill (\x) circle (2pt);
   \foreach \x in {C,v1,v2}
   \fill[blue] (\x) circle (3pt);


\begin{scope}[every coordinate/.style={shift={(10,0,0)}}]


\draw [decorate,decoration={brace,amplitude=5pt},xshift=0.4pt,yshift=-0.4pt] ([c]v1P)--([c]OP) node[black,midway,xshift=0.3cm,yshift=0 cm] {\small{$e$}};
\draw [decorate,decoration={brace,amplitude=5pt},xshift=0.4pt,yshift=-0.4pt] ([c]v2P)--([c]v1P) node[black,midway,xshift=0 cm,yshift=-0.3cm] {\small{$e$}};
\draw [decorate,decoration={brace,amplitude=5pt,mirror},xshift=0.4pt,yshift=-0.1pt] ([c]OP)--([c]v2P) node[black,midway,xshift=0.25cm,yshift=-.2cm] {\small{$e$}};

\draw[->] ([c]OP)--([c]BP);
\draw [->]([c]OP)--([c]CP);

\draw[->] ([c]OP)--([c]vl3);
\draw[->] ([c]v2P)--([c]vl1);
\draw[->] ([c]v1P)--([c]vl2);

   \fill[blue]([c]OP) circle (3pt);
       \fill[blue] ([c]v1P) circle (3pt);
       \fill[blue] ([c]v2P) circle (3pt);
       
       \node at ([c]OP) [below] {\small{$v_0$}};	
     \node at ([c]v1P) [left] {\small{$v_1$}};
     \node at ([c]v2Plab) [below] {\small{$v_2$}};
     
      \node at ([c]BP) [right] {\small{$\ell_1$}};
          \node at ([c]CP) [above] {\small{$\ell_2$}};

	\node at ([c]P01) [left] {\tiny{$P_{01}$}};
	\node at ([c]P02) [rotate=45] {\tiny{$P_{02}$}};
	\node at ([c]P12) [above] {\tiny{$P_{12}$}};

\end{scope}


\end{tikzpicture}
\caption{Nontrivial rubber action on join divisors.}
\label{actionondiv}
\end{figure}

\noindent This produces a family of expansions $\pi \colon \Xcal_{\Upsilon} \to \Aaff^{\!1}$ whose central fibre $Y_\Upsilon \subseteq \Xcal_{\Upsilon}$ has three irreducible components, indexed by the vertices $v_i$. The first is
\[ Y_{v_0} = \Bl_Z X\]
where $Z=D_1 \cap D_2$. This coincides with the general fibre of the family $\pi$, which is the logarithmic modification of $X$ induced by the asymptotic cone complex (we note however that $Y_{v_0}$ and the general fibre do not always coincide). Next we have
\[ Y_{v_1}=\PP_{D_2}(\OO_{D_2}(D_2)\oplus \OO_{D_2})\]
which is a $\PP^1$ bundle over the divisor $D_2$. Finally, $Y_{v_2}$ is the blowup of the $\PP^2$ bundle
\begin{equation*} \label{eqn: P2 bundle} \PP_{Z}(\OO_Z(D_1)\oplus \OO_Z(D_2) \oplus \OO_Z) \to Z \end{equation*}
along two sections of the projection, each of which gives a torus-fixed point when intersected with any fibre over $Z$. Alternatively, the toric variety bundle $Y_{v_2} \to Z$ can be constructed directly as a fibrewise GIT quotient. As explained in \cite{CarocciNabijou2} there are line bundles $L_1,\ldots,L_5$ on $Z$ corresponding to the edges adjacent to $v_2$ in the polyhedral complex, such that $Y_{v_2}$ is a GIT quotient:
\begin{equation*} Y_{v_2} = (L_1 \oplus \cdots \oplus L_5) \sslash \Gm^3.\end{equation*}
Labelling the edges anticlockwise starting from the ray with direction vector $(1,0)$, we can take:
\[  L_1=\OO_Z(D_1),L_3=\OO_Z(D_2),L_2=L_4=L_5=\OO_Z. \]
This furnishes $Y_{v_2}$ with homogeneous coordinates $[x_1,\ldots,x_5]$ relative to the base $Z$.

The join divisors are the strata $Y_{P_{01}},Y_{P_{02}},Y_{P_{12}}$. The first of these is isomorphic to $D_2$ while the other two are $\PP^1$ bundles over $Z$. Note that the inclusion $Y_{P_{02}} \subseteq Y_{v_0}$ is the exceptional divisor of the blowup, and that the inclusion $Y_{P_{12}} \subseteq Y_{v_1}$ is the restriction of the $\PP^1$ bundle $Y_{v_1} \to D_2$ to $Z\subseteq D_2$. 

 The rubber torus is coordinatised by the single tropical parameter $e$:
 \[ T_\tau = \tau \otimes \Gm = \Gm \langle e \rangle. \]
In particular it is one-dimensional, even though the expanded target has three components. This occurs due to dependencies between the tropical edge lengths, forced by fixing the slopes.

 We will describe the rubber action on the locally-closed strata of the target. The action on $Y_{v_0}$ is trivial because $\sigma_{v_0}=0$, as is always the case for the level zero component. On the other hand, the projections
 \[ Y_{v_1}^\circ \to X_{v_1}^\circ = D_2 \setminus Z, \qquad Y_{v_2}^\circ \to X_{v_2}^\circ = Z\]
 are principal torus bundles of relative dimensions $1$ and $2$, respectively. By Theorem~\ref{prop: rubber action general} their structure groups are given by
 \begin{align*} T_{\sigma_1} & = \RR_{\geq 0}\langle \ell_2 \rangle \otimes \Gm = \Gm \langle \ell_2 \rangle, \\
T_{\sigma_2} & = \RR_{\geq 0}^2 \langle \ell_1,\ell_2 \rangle\otimes \Gm = \Gm^2 \langle \ell_1,\ell_2 \rangle
 \end{align*}
 while the tropical position maps are:
 \begin{align*}
 \varphi_{v_1} & \colon \RR_{\geq 0}\langle e \rangle \to \RR_{\geq 0} \langle \ell_2 \rangle, \qquad \ e \mapsto e,\\
 \varphi_{v_2} & \colon \RR_{\geq 0} \langle e \rangle \to \RR_{\geq 0}^2 \langle \ell_1,\ell_2 \rangle, \quad e \mapsto (e,e).
 \end{align*} 
 These determine the action of the rubber torus $ \Gm \langle e \rangle$: it acts on $Y_{v_1}^\circ$ and $Y_{v_2}^\circ$ simultaneously, via a uniform rescaling in the bundle directions. Note in particular that there is no way to split the rubber torus into factors corresponding to the irreducible components of the target.
  
This describes the rubber action on the interior of each component. By continuity, this determines the action on all locally-closed strata. The same principle can also be used to obtain global descriptions of the action on \emph{closed} strata. For instance, on $Y_{v_2}$ the action is given in terms of the homogeneous coordinates $[x_1,\ldots,x_5]$ as:
\[ \lambda \cdot [x_1,x_2,x_3,x_4,x_5] = [\lambda x_1,x_2,\lambda x_3,x_4,x_5]. \]
Again by continuity, we see that the action must fix the join divisor $Y_{P_{02}} \subseteq Y_{v_2}$ pointwise, while acting nontrivially on $Y_{P_{12}} \subseteq Y_{v_2}$. This is consistent with the description of the actions on $Y_{v_0}$ and $Y_{v_2}$, verifying that we indeed have a well-defined action on all of $Y_\Upsilon$.
 
 The action on the join divisors can be described more carefully in terms of generalised tropical position maps. The cones giving the structure groups for the associated principal torus bundles are:
 \begin{align*} (Y_{P_{01}}^\circ \to X_{P_{01}}^\circ = D_2\setminus Z) \colon & \ \ \theta_{P_{01}} = \RR_{\geq 0} \langle \ell_2 \rangle / \RR_{\geq 0} \langle \ell_2 \rangle = 0, \\
 (Y_{P_{02}}^\circ \to X_{P_{02}}^\circ = Z) \colon & \ \ \theta_{P_{02}} = \RR^2_{\geq 0} \langle \ell_1,\ell_2 \rangle / \RR_{\geq 0} \langle \ell_1+\ell_2 \rangle \cong \RR_{\geq 0}, \\
 (Y_{P_{12}}^\circ \to X_{P_{12}}^\circ = Z) \colon & \ \ \theta_{P_{12}} = \RR^2_{\geq 0} \langle \ell_1,\ell_2 \rangle/\RR_{\geq 0} \langle \ell_1 \rangle = \RR_{\geq 0} \langle \ell_2 \rangle.
 \end{align*}
However, the tropical position map $\varphi_{P_{02}} \colon \RR_{\geq 0} \langle e \rangle \to \RR_{\geq 0}$ is zero. This is because the generalised position of $P_{02}$ is obtained by projecting $P_{02}$ away from the diagonal line in $\RR^2_{\geq 0} \langle \ell_1,\ell_2 \rangle$, which always gives $0$ regardless of the choice of $e$. On the other hand the tropical position map for $P_{12}$ is:
\begin{equation*} \varphi_{P_{12}} \colon \RR_{\geq 0} \langle e \rangle \to \RR_{\geq 0} \langle \ell_2 \rangle, \quad e \to e. \end{equation*}
We conclude that the actions on $Y_{P_{01}}$ and $Y_{P_{02}}$ are trivial, while the action on $Y_{P_{12}}$ is given by rescaling the fibres of the $\PP^1$ bundle $Y_{P_{12}} \to Z$.
  
We see in particular that the rubber torus acts nontrivially on a join divisor. As discussed in \S\ref{sec: recursive description of boundary}, this is the main obstruction to obtaining a recursive description of the boundary of the space of expanded stable maps
\end{example}

\begin{example} \label{example: third rubber example} Consider once again a pair $(X| D=D_1+D_2)$ with tropicalisation:
\[ \Sigma = \RR^2_{\geq 0} \langle \ell_1,\ell_2 \rangle. \]
We consider the following tropical expansion, with two-dimensional base cone $\tau = \RR^2_{\geq 0} \langle e_1,e_2 \rangle$:

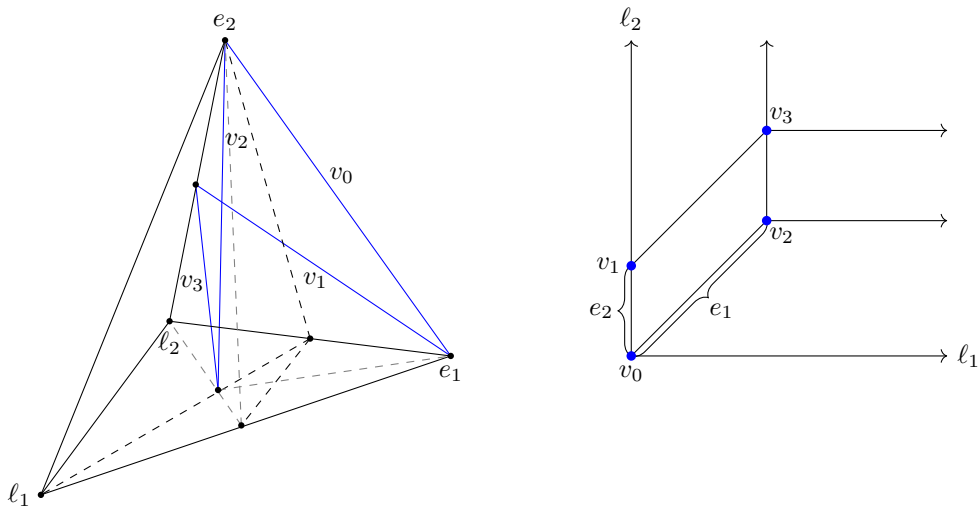
\begin{figure}[htb]
	\centering
		\centering
		\begin{tikzpicture} [scale=.6]
		\tikzstyle{every node}=[font=\normalsize]
		\tikzset{arrow/.style={latex-latex}}
		
		\tikzset{cross/.style={cross out, draw, thick,
         minimum size=2*(#1-\pgflinewidth), 
         inner sep=1.2pt, outer sep=1.2pt}}

\coordinate (O) at (0,0,-2);
 \coordinate (B) at (7,0,0);
 \coordinate(e2l2) at (1.35,3.8,0);
   \coordinate (C) at (2,7,0);
      \coordinate (A) at (1,0,8);
     		
 \coordinate(e1l1) at (3.9,0,4);
  \coordinate(e1l2) at (3.5,0,-1);
   
  \coordinate(e1l1l2) at (3,.4,3);

\begin{scope}[every coordinate/.style={shift={(0,0,0)}}]
          
\draw[blue]([c]e2l2)--([c]B);
\draw[blue]([c]e2l2)--([c]e1l1l2);
   
\draw[blue] ([c]B)--([c]C);
\draw[blue]([c]e1l1l2)--([c]C);

\draw ([c]O)--([c]B);
\draw ([c]O)--([c]C);
\draw([c]A)--([c]O);
\draw([c]A)--([c]B);
\draw([c]A)--([c]C);

\draw[dashed]([c]e1l2)--([c]e1l1);
\draw[dashed]([c]e1l2)--([c]C);
\draw[dashed]([c]e1l2)--([c]A);

\draw[dashed,gray]([c]e1l1)--([c]O);
\draw[dashed,gray]([c]e1l1)--([c]C);
\draw[dashed,gray]([c]e1l1l2)--([c]B);

\foreach \x in {O,B, C,A,e2l2,e1l2,e1l1,e1l1l2}
   \fill ([c]\x) circle (2pt);

\node at ([c]O) [below] {\small{$\ell_2$}};	
     \node at ([c]B) [below] {\small{$e_1$}};
     \node at ([c]C) [above] {\small{$e_2$}};
          \node at ([c]A) [left] {\small{$\ell_1$}};
     
     \node at ([c]C) [xshift=1.55cm,yshift=-1.8cm] {\small{$v_0$}};	
       \node at ([c]C) [xshift=1.2cm,yshift=-3.2cm] {\small{$v_1$}};
               \node at ([c]O) [xshift=.3cm,yshift=.5cm] {\small{$v_3$}};		
              \node at ([c]O) [xshift=0.9cm,yshift=2.4cm] {\small{$v_2$}};	


\end{scope}


\coordinate (OP) at (0,0,0);
\coordinate (BP) at (7,0,0);
   \coordinate (CP) at (0,7,0);
 \coordinate (v1P) at (0,2,0);
      \coordinate (v2P) at (3,3,0);
            \coordinate (v3P) at (3,5,0);
      \coordinate (vl2) at (7,3,0);
      \coordinate (vl31) at (7,5,0);
    \coordinate (vl32) at (3,7,0);

\begin{scope}[every coordinate/.style={shift={(11,0,0)}}]


\draw [decorate,decoration={brace,amplitude=5pt,mirror},xshift=0.4pt,yshift=-0.4pt] ([c]v1P)--([c]OP) node[black,midway,xshift=-0.4cm,yshift=0 cm] {\small{$e_2$}};
\draw [decorate,decoration={brace,amplitude=5pt,mirror},xshift=0.4pt,yshift=0.4pt] ([c]OP)--([c]v2P) node[black,midway,xshift=0.3cm,yshift=-0.3cm] {\small{$e_1$}};

\draw[->] ([c]OP)--([c]BP);
\draw [->]([c]OP)--([c]CP);

\draw ([c]OP)--([c]v2P);
\draw ([c]v2P)--([c]v3P);
\draw ([c]v1P)--([c]v3P);
\draw [->]([c]v2P)--([c]vl2);
\draw [->]([c]v3P)--([c]vl31);
\draw [->]([c]v3P)--([c]vl32);

   \fill[blue]([c]OP) circle (3pt);
       \fill[blue] ([c]v1P) circle (3pt);
       \fill[blue] ([c]v2P) circle (3pt);
          \fill[blue] ([c]v3P) circle (3pt);
       
       \node at ([c]OP) [below] {\small{$v_0$}};	
     \node at ([c]v1P) [left] {\small{$v_1$}};
     \node at ([c]v2P) [xshift=.2cm, yshift=-0.2cm] {\small{$v_2$}};
          \node at ([c]v3P) [xshift=.2cm, yshift=0.2cm] {\small{$v_3$}};

      \node at ([c]BP) [right] {\small{$\ell_1$}};
          \node at ([c]CP) [above] {\small{$\ell_2$}};

\end{scope}

\end{tikzpicture}
\caption{Rank two target and base.}
\label{fig:rank2rubber2} 
\end{figure}

\noindent The polyhedral subdivision has two independent tropical parameters $e_1,e_2$. These coordinatise the base cone $\tau$ and the rubber torus:
\[ T_\tau = \Gm^2 \langle e_1,e_2 \rangle. \]
As before let $Z=D_1 \cap D_2$. The structure groups of the principal torus bundles $Y_{v_i}^\circ \to X_{v_i}^\circ$ are:
\begin{align*} 
(Y_{v_1}^\circ \to X_{v_1}^\circ = D_2\setminus Z) \colon \ \ & T_{\sigma_1} = \Gm \langle \ell_2 \rangle, \\
(Y_{v_2}^\circ \to X_{v_2}^\circ = Z) \colon \ \ & T_{\sigma_2} = \Gm^2 \langle \ell_1,\ell_2 \rangle, \\
(Y_{v_3}^\circ \to X_{v_3}^\circ = Z) \colon \ \ & T_{\sigma_3} = \Gm^2 \langle \ell_1,\ell_2 \rangle.
\end{align*}
The tropical position maps are given by:
\begin{align*}
\varphi_{v_1}(e_1,e_2) & = e_2, \\
\varphi_{v_2}(e_1,e_2) & = (e_1,e_1), \\
\varphi_{v_3}(e_1,e_2) & = (e_1,e_1+e_2).	
\end{align*}
These induce homomorphisms from the rubber torus to the structure groups of the principal torus bundles $Y_{v_i}^\circ \to X_{v_i}^\circ$. These determine the rubber action on the interior of each component:
\begin{align*}
\varphi_{v_1} \otimes \Gm & \colon \Gm^2 \langle e_1,e_2 \rangle \to \Gm \langle \ell_2 \rangle, \qquad \ (e_1,e_2) \mapsto e_2,\\
\varphi_{v_2} \otimes \Gm & \colon \Gm^2 \langle e_1,e_2 \rangle \to \Gm^2 \langle \ell_1,\ell_2 \rangle, \quad (e_1,e_2) \mapsto (e_1,e_1), \\
\varphi_{v_3} \otimes \Gm & \colon \Gm^2 \langle e_1,e_2 \rangle \to \Gm^2 \langle \ell_1,\ell_2 \rangle, \quad (e_1,e_2) \mapsto (e_1,e_1 e_2).
\end{align*}
The rubber action is nontrivial on the join divisors $Y_{v_1}\cap Y_{v_3}$ and $Y_{v_2} \cap Y_{v_3}$ as can be seen by considering generalised tropical position maps.
\end{example}

\begin{example} Consider the following modification of the previous example

\begin{figure}[H]
\centering
\begin{tikzpicture} [scale=.6]

\coordinate (OP) at (0,0,0);
\coordinate (BP) at (7,0,0);
   \coordinate (CP) at (0,7,0);
 \coordinate (v1P) at (0,3,0);
      \coordinate (v2P) at (3,3,0);
            \coordinate (v3P) at (3,6,0);
      \coordinate (vl2) at (7,3,0);
      \coordinate (vl31) at (7,6,0);
    \coordinate (vl32) at (3,7,0);
    \coordinate (Pnew) at (1.5,2.8,0);

\begin{scope}[every coordinate/.style={shift={(11,0,0)}}]


\draw [decorate,decoration={brace,amplitude=5pt,mirror},xshift=0.4pt,yshift=-0.4pt] ([c]v1P)--([c]OP) node[black,midway,xshift=-0.4cm,yshift=0 cm] {\small{$e$}};
\draw [decorate,decoration={brace,amplitude=5pt,mirror},xshift=0.4pt,yshift=0.4pt] ([c]OP)--([c]v2P) node[black,midway,xshift=0.3cm,yshift=-0.3cm] {\small{$e$}};

\draw[->] ([c]OP)--([c]BP);
\draw [->]([c]OP)--([c]CP);

\draw ([c]OP)--([c]v2P);
\draw ([c]v2P)--([c]v3P);
\draw ([c]v1P)--([c]v3P);
\draw ([c]v3P)--([c]OP);
\draw [->]([c]v2P)--([c]vl2);
\draw [->]([c]v3P)--([c]vl31);
\draw [->]([c]v3P)--([c]vl32);

   \fill[blue]([c]OP) circle (3pt);
       \fill[blue] ([c]v1P) circle (3pt);
       \fill[blue] ([c]v2P) circle (3pt);
          \fill[blue] ([c]v3P) circle (3pt);
       
       \node at ([c]OP) [below] {\small{$v_0$}};	
     \node at ([c]v1P) [left] {\small{$v_1$}};
     \node at ([c]v2P) [xshift=.2cm, yshift=-0.2cm] {\small{$v_2$}};
          \node at ([c]v3P) [xshift=.2cm, yshift=0.2cm] {\small{$v_3$}};
          
          \node at ([c]Pnew) [left] {\small{$P$}};

      \node at ([c]BP) [right] {\small{$\ell_1$}};
          \node at ([c]CP) [above] {\small{$\ell_2$}};

\end{scope}
\end{tikzpicture}
\caption{Further subdivision changes the rubber torus.}
\end{figure}
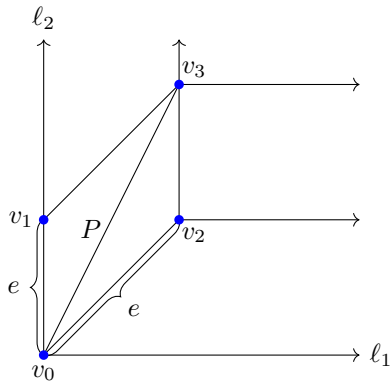
\noindent where we have introduced an additional diagonal edge $P$ with slope $(1,2)$. This forces an identification of the previously independent tropical parameters:
\begin{equation*} e = e_1 = e_2. \end{equation*}
Consequently the base cone, and hence the rubber torus, changes as well, from $\Gm^2\langle e_1, e_2 \rangle$ to the diagonal subtorus:
\[ \Gm \langle e \rangle \hookrightarrow \Gm^2 \langle e_1,e_2 \rangle.\]
The new rubber action arises by composing this inclusion with the previous action maps $\varphi_{v_i} \otimes \Gm$. We therefore see that refining a polyhedral subdivision induces a reduction of the rubber torus. This is explained by the fact that $\Gm \langle e \rangle$ is precisely the subtorus of $\Gm^2 \langle e_1,e_2 \rangle$ which acts trivially on the new join divisor $Y_P$.
\end{example}

\begin{example} The previous examples all concern tropical expansions which are complete. As discussed in Remark~\ref{rmk: T not complete}, in enumerative geometry it is often more convenient to work with non-complete expansions. Typically we are interested only in the image of a given tropical curve, which amounts to considering expansions built exclusively from $0$- and $1$-dimensional polyhedra. Below is a variant of Example~\ref{example: third rubber example} demonstrating an instance of a non-complete expansion. 

\begin{figure}[H]
	\centering
		\centering
		\begin{tikzpicture} [scale=.6]
		\tikzstyle{every node}=[font=\normalsize]
		\tikzset{arrow/.style={latex-latex}}
		
		\tikzset{cross/.style={cross out, draw, thick,
         minimum size=2*(#1-\pgflinewidth), 
         inner sep=1.2pt, outer sep=1.2pt}}

\coordinate (OP) at (0,0,0);
\coordinate (BP) at (7,0,0);
   \coordinate (CP) at (0,7,0);
 
      \coordinate (v2P) at (3,3,0);
            \coordinate (v3P) at (3,5,0);
      \coordinate (vl2) at (7,3,0);
      \coordinate (vl31) at (7,5,0);
    \coordinate (vl32) at (3,7,0);

\begin{scope}[every coordinate/.style={shift={(0,0,0)}}]

\fill [orange!20] (0.2,0.4) -- (2.8,3) -- (2.8,7) -- (0.2,7) -- (0.2,0.4);
\fill [orange!20] (0.4,0.2) -- (7,0.2) -- (7,2.8) -- (3,2.8) -- (0.4,0.2);
\fill [orange!20] (3.2,3.2) -- (7,3.2) -- (7,4.8) -- (3.2,4.8) -- (3.2,3.2);
\fill [orange!20] (3.2,5.2) -- (7,5.2) -- (7,7) -- (3.2,7) -- (3.2,5.2);


\draw [decorate,decoration={brace,amplitude=5pt,mirror},xshift=0.4pt,yshift=-0.4pt] ([c]v2P)--([c]v3P) node[black,midway,xshift=0.4cm,yshift=0 cm] {\small{$e_2$}};
\draw [decorate,decoration={brace,amplitude=5pt,mirror},xshift=0.4pt,yshift=0.4pt] ([c]OP)--([c]v2P) node[black,midway,xshift=0.3cm,yshift=-0.3cm] {\small{$e_1$}};
\draw[->] ([c]OP)--([c]BP);
\draw [->]([c]OP)--([c]CP);

\draw ([c]OP)--([c]v2P);
\draw ([c]v2P)--([c]v3P);

\draw [->]([c]v2P)--([c]vl2);
\draw [->]([c]v3P)--([c]vl31);
\draw [->]([c]v3P)--([c]vl32);

   \fill[blue]([c]OP) circle (3pt);

       \fill[blue] ([c]v2P) circle (3pt);
          \fill[blue] ([c]v3P) circle (3pt);
       
       \node at ([c]OP) [below] {\small{$v_0$}};	
     \node at ([c]v2P) [xshift=.2cm, yshift=-0.2cm] {\small{$v_2$}};
          \node at ([c]v3P) [xshift=.2cm, yshift=0.175cm] {\small{$v_3$}};

      \node at ([c]BP) [right] {\small{$\ell_1$}};
          \node at ([c]CP) [above] {\small{$\ell_2$}};

\end{scope}
	\end{tikzpicture}

\caption{Non-complete expansion.}
\label{fig:noncompact} 
\end{figure}
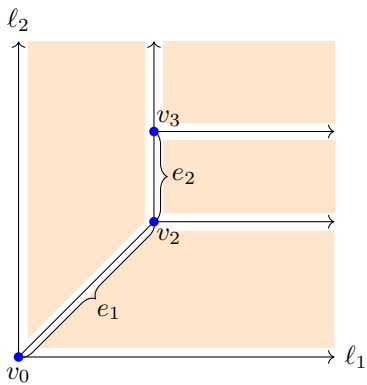

\noindent The orange shaded regions are not included in the conical or polyhedral complex. Note that there is no canonical choice of compactification here. The rubber torus
\[ T_\tau = \Gm^2 \langle e_1, e_2 \rangle \]
is the same as before, as is the description of its action on $Y_{v_2}^\circ$ and $Y_{v_3}^\circ$. We see that complete and non-complete expansions can be handled uniformly.
\end{example}

\footnotesize
\bibliographystyle{alpha}
\bibliography{Bibliography.bib}\medskip

\noindent Francesca Carocci: University of Geneva, Mathematics Department, Rue du Conseil-Général 7-9, CH-1204 Genève, Switzerland. Email: \href{mailto:francesca.carocci@unige.ch}{francesca.carocci@unige.ch}

\noindent Navid Nabijou: School of Mathematical Sciences, Queen Mary University of London, London E1 4NS, United Kingdom. Email: \href{mailto:n.nabijou@qmul.ac.uk}{n.nabijou@qmul.ac.uk}

\end{document}